\documentclass[10pt]{article}
\usepackage{amsfonts,amssymb}
\usepackage{amsmath}
\usepackage{amsthm}
\usepackage[matrix,arrow,curve]{xy}
\usepackage[left=3cm,right=3cm,top=3cm,bottom=3cm]{geometry}
\usepackage{hyperref}
\hypersetup{linktoc=page}
\usepackage[capitalize, noabbrev]{cleveref}
\usepackage{enumitem}

\setitemize[0]{label=---}

\newcommand{\p}{\partial}
\newcommand{\e}{\varepsilon}
\newcommand{\C}{{\mathbb C}}
\newcommand{\R}{{\mathbb R}}
\newcommand{\T}{{\mathbb T}}
\newcommand{\N}{{\mathbb N}}
\newcommand{\Z}{{\mathbb Z}}
\newcommand{\M}{{\mathcal M}}
\newcommand{\F}{\mathcal F}
\newcommand{\Id}{\text{Id}}
\newcommand{\ov}{\overline}

\newcommand{\la}{\langle}
\newcommand{\ra}{\rangle}
\newcommand{\ad}{\text{ad}}
\newcommand{\Ad}{\text{Ad}}
\newcommand{\f}{\overrightarrow}
\newcommand{\E}
{\overrightarrow{\exp}\int}
\newcommand{\we}{\wedge}

\newtheorem{theorem}{Theorem}
\newtheorem{lemma}{Lemma}
\newtheorem{prop}{Proposition}

\newtheorem{definition}{Definition}

\newtheorem{rem}{Remark}

\title{Approximate controllability on the group of volume-preserving diffeomorphisms}
\author{Andrei~Agrachev\thanks{SISSA, Trieste, agrachev@sissa.it} \and Bettina~Kazandjian\thanks{Sorbonne Université, Université Paris Cité, CNRS, Inria, Laboratoire Jacques Louis Lions, Paris, bettina.kazandjian@sorbonne-universite.fr}} 
\date{}
\begin{document}

\maketitle

\begin{abstract} We study controlability issues for the group of volume-preserving diffeomorphisms of the torus $\mathbb T^d$ for system
$\dot x=f(x)+u(t)$, where $f$ is a fixed divergence free vector field on $\mathbb T^d$ and $u(t)$ are constant vector fields which generate translations of the torus. Main results concern $d$ equals two or three.
\end{abstract}

\tableofcontents

\section{Introduction} Let $\mathbb T^d=\mathbb R^d/2\pi\mathbb Z^d$ be the $d$-dimensional torus. Translations of the torus are generated by constant vector fields. Let $f$ be a non constant divergence free vector field on $\mathbb T^d$: it generates a one-parametric group of
volume-preserving diffeomorphisms. In this paper, we try to understand which transformations of $\mathbb T^d$ can be reached if we perturb $f$
by constant fields (with the constant depending on time), mainly for $d=2$ and $d=3$.  

The answer is surprisingly simple in many cases and we hope that it might be useful in the mathematical fluid dynamics, in particular, for
the study of turbulent flows. In the next paper, we are going to treat random perturbations.

We start with a general linear in control system $\dot x=\sum\limits_{j=1}^su_jF_j(x)$ on a compact manifold $M$. \cref{theorem : motion planning diffeo} characterizes
the flows on $M$ which can be uniformly approximated by the flows generated by a time-varying ordinary differential equation of the form
 $\dot x=\sum\limits_{j=1}^su_j(t)F_j(x)$. These are exactly the flows generated by equations $\dot x=V_t(x)$, where 
 $V_t\in\overline{\mathrm{Lie}\{F_1,\ldots,F_s\}}$.
 
In \cref{theorem : controllability ensembles} we deal with the flows which preserve a fixed volume form on $M$. Such flows are generated by the divergence free vector fields.
Let $\mathrm{Vec}_0M$ be the space of divergence free vector fields.  \cref{theorem : controllability ensembles} states that the condition 
$\mathrm{Vec}_0M\subset\overline{\mathrm{Lie}\{F_1,\ldots,F_s\}}$ is sufficient for the possibility to transfer any finite ensemble of mutually
distinct points $x_i\in M,\ i=1,\ldots,N$, to any other sequence $y_i\in M,\ i=1,\ldots,N$, by the flow generated by the equation of the form
$\dot x=\sum\limits_{j=1}^su_j(t)F_j(x)$, where the control $(u_1(\cdot),\ldots,u_s(\cdot))$ and the transfer time are the same for all 
$x_i,\ i=1,\ldots,N$.

General theorems 1--2 are rather simple corollaries of earlier results. We need them to study the affine in control system $\dot x=f(x)+u$
that is our main objective. Here $u$ is an arbitrary constant vector field on the torus $\mathbb T^d$.

Let $f(x)=\sum\limits_{m\in\mathbb Z^d}p_me^{i\langle m,x\rangle}$ and $\M_f=\left\{m\in\Z^d \mid p_m\neq 0\right\}$. \cref{theorem : Lie algebra T2 and T3} and \cref{theorem : affine system approx diffeo}
concern the case where $\#\M_f<\infty$ and $\mathrm{span}\{f(x)\mid x\in\mathbb T^d\}=\mathrm{span}\M_f=\mathbb R^d$.

In \cref{theorem : Lie algebra T2 and T3} we describe the closure of the Lie subalgebra of $\mathrm{Vec}_0\mathbb T^d$ generated by the fields $f+u$ for $d$ equals two or three. This closure is equal
to the space of fields $g\in\mathrm{Vec}_0\mathbb T^d$ such that $\M_g$ is contained in the subgroup of $\mathbb Z^d$ generated by $\M_f$.

In \cref{theorem : affine system approx diffeo} we describe the closure of the attainable set of the system $\dot x=f(x)+u$ in the group of volume-preserving diffeomorphisms
of $\mathbb T^d$. It appears that the closure of the attainable set depends only on the Lie algebra computed in \cref{theorem : Lie algebra T2 and T3} in spite of the
fact that our system is not linear in control, it has a nontrivial drift $f$ while its linear part $u$ is commutative.

Finally, \cref{theorem : affine system ensemble controllability} concerns the controllability for finite ensembles of points on $\mathbb T^d$. Here we do not need the set $\M_f$ to be finite. The theorem states that for $d$ equals two or three there exists a residual subset of  $\mathrm{Vec}_0\mathbb T^d$ such that the system
$\dot x=f(x)+u$ can transfer any finite ensemble of mutually distinct points to any other ensemble with the same number of points.
On the other hand, according to the same theorem, whatever $f$ we take, the transfer of at least two points ensemble may require arbitrary long time. No a priori time bound is possible.

These are our main results. In \cref{section : Lie algebra computation T2 T3}, the Lie algebra generated by $f+u,\ u\in\mathbb R^d$, is computed also in the case where $\mathrm{span}\M_f\ne\mathbb R^d$ (see \cref{theorem : algebra dim2} and \cref{theorem : appendix T3 span1}).

 \begin{rem} Admissible controls in this paper are measurable locally bounded vector functions but all results remain valid with much smaller classes of admissible controls. Indeed, we say that a sequence of locally integrable vector functions $u_n(\cdot)$ converges to $u(\cdot)$ on the segment
$[0,T]$ in the \it{sliding topology} if $\|u_n\|_{L^1}$ are uniformly bounded and $\left|\int_0^tu_n(t)-u(t)\,dt\right|\to 0$ uniformly for
$t\in[0,T]$ as $n\to\infty$.

The map which sends the control $u(\cdot)$ to the flow $P^{u(\cdot)}_\cdot$ is continuous in the sliding topology (see, for instance, Lemma~8.10
in \cite{AgSa}) and all our results survive if we use only controls from an everywhere dense subset in the sliding topology. For example, it is
sufficient to use only controls from the space of vector polynomials or trigonometric polynomials, or piecewise constant vector functions, or
even from the set of piecewise constant vector functions with only one nonzero coordinate in every time moment.
\end{rem}  

To conclude the introduction, we have to mention that various aspects of the controllability on the group of diffeomorphisms were studied in
\cite{AgCa,ArTr,BePo25(1),BePo25(2),BeTu,DuNer,KaPoSi,Ner,PoSca,Sca} and many other papers.

\section{Motion planning in the group of diffeomorphisms and controllability of finite ensembles of points}

Let $M$ be a compact Riemannian manifold of class $C^{\infty}$ and dimension $n \in \N$. Let us consider the linearly controlled equation
\begin{equation} \label{control system 1}
    \dot{x}=\sum_{j=1}^su_j(t)F_j(x), \qquad u(t)=(u_1(t),\dots,u_s(t))\in \R^s,x\in M,
\end{equation}
where the measurable map $t\mapsto u(t)$ is locally bounded. The flow of \eqref{control system 1} at time $t$ is denoted by $P^{u(\cdot)}_t$. Considering the family of smooth vector fields $\F:=\left\{F_1,\dots,F_s\right\}$, we would like to understand which trajectories in the group of diffeomorphisms of $M$ could be approximated by the flows $t \mapsto P^{u(\cdot)}_t$, uniformly on any time segment. This is a problem of motion planning in the group of diffeomorphisms. 

Let us consider $N \in \N^*$ distinct points of $M$, $\gamma=(\gamma_1,\dots,\gamma_N)\in \hat{M}^N$, where $\hat{M}^N=M^N\setminus \Delta^N$ and $\Delta^N:=\left\{\gamma \in M^N \mid \exists k \neq \ell, \gamma_k=\gamma_\ell \right\}$. If we apply the dynamic \eqref{control system 1} to the initial positions $(\gamma_1, \dots, \gamma_N)$, the configuration of these points at time $t \geq 0$ is determined by $(P_t^{u(\cdot)}(\gamma_1),\dots,P_t^{u(\cdot)}(\gamma_N))$. We would like to study properties about controllability of finite ensembles of points depending on $\F$.

\subsection{Motion planning in the group of diffeomorphisms}

In what follows, $M$ is a compact Riemannian manifold of class $C^{\infty}$. Given a smooth tensor field $q\mapsto\omega_q,\ q\in M$,
where $\omega_q\in(T_qM)^{\otimes k}\otimes(T^*_qM)^{\otimes \ell}$, the norms $\|\omega\|_m,\ m=0,1,2,\ldots,$ are defined as follows:
$$
\|\omega\|_m=\sup\limits_{q\in M,i\le m}|\nabla^i_q\omega|,
$$
where $\nabla$ is the covariant derivative. Here $k=1,2,\ldots,\ \ell=0,1,2,\ldots,$ and $(T^*M)^{\otimes 0}=C^{\infty}(M)$. The seminorms $\|\cdot\|_m$
define standard $C^\infty$-topology on the space of smooth tensor fields of the prescribed degree. Standard $C^\infty$-topology on the group of 
diffeomorphisms is induced by the topology on $C^\infty(M)$ if we treat a diffeomorphism $P:M\to M$ as a linear operator $P^*$ on 
$C^\infty(M)$, where $P^*:a\mapsto a\circ P, \ \|P\|_m=\sup\limits_{\|a\|_m\le 1}\|P^*a\|_m$. More details about this formalism and chronological calculus can be found in \cite[Chap.~6]{AgBaBo}. 

Let $\mathrm{Vec}M$ be the Lie algebra of smooth vector fields on $M$. Given $\mathcal F\subset\mathrm{Vec}M,\ 
\mathrm{Lie}\mathcal F\subset\mathrm{Vec}M$ is the Lie subalgebra generated by $\mathcal F$ and $\overline{\mathrm{Lie}\mathcal F}$ is the closure
of $\mathrm{Lie}\mathcal F$ in the standard topology.

A measurable map $t\mapsto V_t$, where $V_t\in\mathrm{Vec}M,\ t\in\R,$ is called a time-varying vector field if $\|v_t\|_m$ is a locally bounded function
of $t$ for any $m\ge 0$. Any time-varying vector field defines a flow $P_t\in\mathrm{Diff}M,\ t\in\R,$ where $P_0=\Id,\ 
\frac{\p}{\p t}P_t(x)=V_t(P_t(x)).$ We use the standard chronological notation for this flow, $P_t=\E_0^tV_\tau\,d\tau$.
The following result is a corollary of \cite[Th.~3]{AgSa20}.

\begin{theorem} \label{theorem : motion planning diffeo}
Let $\mathcal F=\left\{F_1,\dots,F_s\right\}\subset\mathrm{Vec}M$. Let $t\mapsto V_t\in\overline{\mathrm{Lie}\mathcal F},\ t\in[0,T]$, be a time-varying vector field. Then for every $m\in \N$ and $\e>0$ there exists a control $t \mapsto u(t)=(u_1(t),\dots,u_s(t)) \in L^\infty([0,T],\R^s)$, such that the flow $P^{u(\cdot)}_t$, generated by control system \eqref{control system 1} and control $u(\cdot)$, satisfies
\begin{equation*}
    \|\E_0^tV_\tau\,d\tau- P^{u(\cdot)}_t \|_m \le\e,\qquad t\in[0,T].
\end{equation*}
\end{theorem}
\begin{proof}
    If $V_t=\sum\limits_{i=1}^kv_i(t)X_i$, where $X_i\in\mathrm{Lie}\mathcal F$, then the desired result is just the
statement of Theorem~3 from \cite{AgSa20}. On the other hand, we can uniformly and arbitrarily well approximate $V_t$  by such linear combination in
any norm $\|\cdot\|_m$. Indeed, $\{V_t\mid \ 0\le t\le T\}$ is a precompact set in the topology $\|\cdot\|_m$ since this set is bounded in the norm     
$\|\cdot\|_{m+1}$. In other words, for every $\delta>0$ there exists a finite set $\{X_1,\dots,X_k\}\subset \mathrm{Lie}\mathcal F$ such that 
the union of the radius $\delta$ balls in norm $\|\cdot\|_m$ centered at $X_i$ covers the set $\{V_t \mid \ 0\le t\le T\}$. We present 
$[0,T]$ as the disjoint union of subsets $S_i,\ [0,T]=\bigcup\limits_iS_i$, such that $\|X_i-V_t\|_m\le\delta$, for every $t\in S_i$, and we set
\begin{equation*}
    v_i(t)=\begin{cases}1& t\in S_i\\ 0& t\notin S_i, \end{cases}
\end{equation*}   
then $\|V_t-\sum\limits_{i=1}^kv_i(t)X_i\|_m\le\delta$, for every $t\in[0,T]$.
\end{proof}  

Let us consider control system \eqref{control system 1}. The connected component of the identity in $\mathrm{Diff} M$ is denoted by $\mathrm{Diff}^0 M$. For every measurable and locally bounded control $u(\cdot)$, the vector field $t \mapsto \sum_{j=1}^su_j(t)F_j$ is a time-varying vector field according to the previous definition.

\begin{definition} \label{def : attainable set diffeo}
    An element $P \in \mathrm{Diff}^0M$ is said to be approximately reachable for system \eqref{control system 1} in time $t\geq 0$ if for every $m\in \N$ and $\e >0$, there exists a measurable and locally bounded control $u(\cdot)$ such that the flow $P^{u(\cdot)}_t$, generated by system \eqref{control system 1} and control $u(\cdot)$, satisfies
    \begin{equation*}
        \|P-P^{u(\cdot)}_t\|_{m}\leq \e.
    \end{equation*}
    The set of approximately reachable diffeomorphisms in time $t\geq 0$ is denoted by $\ov{\mathcal{A}}_t$ and $\ov{\mathcal{A}}:=\cup_{t\geq 0}\ov{\mathcal{A}_t}$. For every subgroup $D \subset \mathrm{Diff}^0M$, system \eqref{control system 1} is said to be approximately controllable in $D$ if $\ov{\mathcal{A}}=D$.
\end{definition}
 Let $\omega$ be a fixed volume form on $M$ and $\mathrm{Vec}_0M$ be the Lie algebra of divergence free vector fields of $M$,
  $$
  \mathrm{Vec}_0M=\{f\in\mathrm{Vec}M \mid  \mathrm{div}_\omega f=0\}.
  $$
  Any volume-preserving \textit{flow} $P_t$ has a form $P_t =\E_0^tf_{\tau}d\tau$, where $t \mapsto f_{\tau}\in \mathrm{Vec}_0M$. Then $P_t\in \mathrm{Diff}^0M$. The set of volume-preserving flows of $M$ is denoted by $\mathrm{Diff}_0M$. This is a subgroup of the connected component of the identity $\mathrm{Diff}^0 M$, and $\mathrm{Diff}_0M\subset \mathrm{Diff}^0M$. 
\begin{prop}
    Let $\F=\left\{F_1,\dots,F_s\right\}$ be the family of admissible vector fields for system \eqref{control system 1}. 
    \begin{itemize}
        \item If $\ov{\mathrm{Lie}\F}=\mathrm{Vec}M$, system \eqref{control system 1} is approximately controllable in $\mathrm{Diff}^0 M$.
        \item If $\ov{\mathrm{Lie}\F}\supset\mathrm{Vec}_0M$, system \eqref{control system 1} is approximately controllable in $\mathrm{Diff}_0 M$.
    \end{itemize}
\end{prop}
\begin{proof}
    This is a corollary of \cref{theorem : motion planning diffeo}.
\end{proof}

\subsection{Controllability of finite ensembles of points}

In what follows, we study the controllability of finite ensembles of points in $\hat{M}^N$. Indeed system \eqref{control system 1} can be lifted to a linear in control system defined on $\hat{M}^N$ by the controlled equations
\begin{equation} \label{control system 1 ensembles}
    \dot{\gamma}_\ell=\sum_{j=1}^s u_j(t)F_j(\gamma_\ell), \qquad u(t)\in \R^s,\quad \ell \in \left\{1,\dots,N\right\},
\end{equation}
where $(\gamma_1,\dots,\gamma_N)\in \hat{M}^N$ and the map $t\mapsto u(t)$ is measurable and locally bounded. The attainable set at time $t\geq 0$ from $\gamma=(\gamma_1,\dots,\gamma_N) \in \hat{M}^N$ of system \eqref{control system 1 ensembles} is defined by 
\begin{equation*}
A_\gamma(t):=\left\{(P_t^{u(\cdot)}(\gamma_1),\dots,P_t^{u(\cdot)}(\gamma_N))\mid u(\cdot) \in L^\infty([0,t],\R^s)\right\} \subset \hat{M}^N.
\end{equation*}

\begin{definition}
    For a general system of control we also define the attainable set from $\gamma$ by $A_\gamma=\cup_{t\geq 0}A_\gamma(t)$ (which coincides with $A_\gamma(t)$ for every $t \geq 0$ in the case of system \eqref{control system 1 ensembles}). Then a system is said globally controllable (respectively globally controllable in time $T\geq 0$) if $A_\gamma=\hat{M}^N$ (respectively if $\cup_{0\leq t \leq T}A_\gamma(t)=\hat{M}^N$) for every $\gamma \in \hat{M}^N$.
\end{definition}
\begin{definition}
    Let $N \in \N^*$. System \eqref{control system 1} is said to be globally controllable (respectively globally controllable in time $T\geq 0$) in the space of $N$-ensembles if system \eqref{control system 1 ensembles} is globally controllable (respectively globally controllable in time $T$) in $\hat{M}^N$.
\end{definition}
The space $\hat{M}^N$ has a structure of smooth manifold. For each $\gamma \in \hat{M}^N$, the tangent space $T_\gamma \hat{M}^N$ is isomorphic to $T_{\gamma_1}M\times \dots \times T_{\gamma_N}M$. The $N$-fold of a vector field $X \in \mathrm{Vec}M$ is defined on $\hat{M}^N$ by $X^{N}(\gamma_1,\dots,\gamma_N)=(X(\gamma_1),\dots,X(\gamma_N))$. If $X$ is complete on $M$ then $X^N$ is also complete on $\hat{M}^N$. The Lie bracket of $N$-folds $X^N,Y^N$ verifies the formula $[X^N,Y^N]=[X,Y]^N$ and the same holds for the iterated Lie brackets.

\begin{definition}
    Let $\F^N=\left\{F_1^N,\dots,F^N_s\right\}$. System \eqref{control system 1 ensembles} is said to be Lie bracket generating at $\gamma$ if $\left\{F(\gamma)\mid F\in \mathrm{Lie}\F^N\right\}=T_{\gamma}\hat{M}^N$. It is Lie bracket generating if it is Lie bracket generating at every $\gamma \in \hat{M}^N$.
\end{definition}

As a consequence of Rashevsky-Chow theorem, if system \eqref{control system 1 ensembles} is Lie bracket generating then it is globally controllable, see e.g. \cite[Th.~5.2 and Cor~5.2]{AgSa}.

\begin{theorem} \label{theorem : controllability ensembles}
    Let $\F=\left\{F_1,\dots,F_s\right\}$. If $\mathrm{Vec}_0M\subset \ov{\mathrm{Lie}\F}$, then the family $\F^N$ is Lie bracket generating in $\hat{M}^N$ for every $N\in \N^*$, and system \eqref{control system 1} is globally controllable in the space of $N$-ensembles.
\end{theorem}
\begin{proof}
    Let $\gamma\in M^N$, we consider the linear map
    \begin{equation*}
        \varphi_{\gamma} : \left\{\begin{array}{ll}
               \mathrm{Lie}\F
               &\rightarrow T_{\gamma} \hat{M}^N \\
              X &\mapsto X^N(\gamma),
        \end{array}\right.
    \end{equation*}
    if it is surjective for every $\gamma \in \hat{M}^N$, then system \eqref{control system 1 ensembles} is Lie bracket generating and so it is globally controllable. By assumption, if $\mathrm{Im}\varphi_\gamma$ denotes the image of $\varphi_\gamma$, then $\left\{X^N(\gamma) \mid X \in  \mathrm{Vec}_0M\right\}\subset \ov{\mathrm{Im}\varphi_\gamma}$.
    
    Recall that $X \in \mathrm{Vec}_0M$ if and only if $\mathrm{div}X=0$. Let us prove that for every $(a_1,\dots,a_N)\in T_{\gamma}\hat{M}^N$, there exists $X\in \mathrm{Vec}_0M$ such that $X^N(\gamma)=(a_1,\dots,a_N)$. Let $\mathcal{V}_1,\dots,\mathcal{V}_N$ be open neighborhoods in $M$ such that
    \begin{equation*}
        \gamma_\ell\in \mathcal{V}_\ell, \quad \mathcal{V}_k \cap \mathcal{V}_\ell = \emptyset, \qquad k\neq\ell \in \left\{1,\dots,N\right\},
    \end{equation*}
    and such that $\mathcal{V}_\ell$ is diffeomorphic to some open neighborhood $\mathcal{O}_\ell \subset \R^n, 0\in \mathcal{O}_\ell$. Locally, the vector field $X$ can be expressed in coordinates. The charts $\phi_\ell : \mathcal{V}_\ell \rightarrow \mathcal{O}_\ell$ are chosen such that the expression of the volume form $\omega$ in coordinates is equal to $dx_1\we\dots \we dx_n$. Given $a=(a_1,\dots,a_n)\in \R^n$ and two neighborhoods $\mathcal{O} \subset \mathcal{O}'$ of $0$ in $\R^n$, we construct $\bar{X} \in \mathrm{Vec}\R^n$ such that $\bar{X}=a$ on $\mathcal{O}$, $\mathrm{supp}\bar{X}\subset \mathcal{O}'$ and $\mathrm{div}\bar{X}=0$. Let $\chi : \R^n \rightarrow \R$ be a smooth cut-off function such that $\chi(x)=\frac{1}{n-1}$ on $\mathcal{O}$ and $\mathrm{supp}\chi\subset \mathcal{O}'$. We consider a $(n-2)$-differential form on $\R^n$,
    \begin{equation*}
        \alpha = \chi \sum_{1\leq \ell < m \leq n}((-1)^{m-2}a_\ell x_m + (-1)^{\ell-1}a_m x_\ell)dx_1 \we \dots \widehat{dx}_\ell \dots \widehat{dx}_m \dots \we dx_n.
    \end{equation*}
    Then we compute $d\alpha =\sum_{m=1}^n \psi_m dx_1\dots\widehat{dx}_m\dots dx_n$ and we check that $\psi_m=a_m$ on $\mathcal{O}$ and $\mathrm{supp} \psi_m\subset \mathcal{O}'$. Let us consider the vector field $\bar{X}=\sum_{m=1}^n\psi_m \p_{x_m}$, then $\bar{X}=a$ on $\mathcal{O}$ and $\mathrm{supp}\bar{X}\subset \mathcal{O}'$. Moreover, $dd\alpha=0=(\mathrm{div}\bar{X})dx_1 \we \dots \we dx_n$, so $\mathrm{div}\bar{X}$=0. The image of $\varphi_{\gamma}$ is dense in $T_{\gamma}\hat{M}^N$, so the map is surjective.
\end{proof}

\section{Volume-preserving diffeomorphisms on $\T^d$}

In the following, we consider the torus $\T^d=\R^d/2\pi\Z^d$. Vector fields on $\T^d$ are naturally identified with $2\pi$-periodic $d$-vector functions on $\R^d$, i.e. the vector function $f(x)=(f^1(x),\ldots,f^d(x)), x=(x_1,\ldots,x_d)\in\R^d,$ corresponds to the field $\sum\limits_{i=1}^df^i(x)\frac\p{\p x_i}$. In the following, we study an affine in control system of the form
\begin{equation} \label{affine system} 
    \dot{x}=f(x)+u(t),  \qquad u(t) \in \R^d, 
\end{equation}
where $f \in \mathrm{Vec}_0 \T^d$ is any divergence free vector field and $t \mapsto u(t)$ is measurable and locally bounded.
\begin{rem}\label{int0}
    By replacing $f$ and $u$ by $f+c$ and $u-c$ where $c\in \R^d$ is a constant, we can suppose that $\int_{\T^d}f=0$ without changing the set $\overline{\mathcal{A}}$.
\end{rem}
The flow at time $t$ of system \eqref{affine system} is denoted by $P^{u(\cdot)}_t$. The set of approximately reachable elements in the group of diffeomorphisms is denoted by $\ov{\mathcal{A}}$, see \cref{def : attainable set diffeo}. We would like to understand which volume-preserving diffeomorphisms could be approximated by the flows of the previous equation, depending on the modes of the Fourier decomposition of $f$. In the following, we study a classification of the approximately reachable set in the group of diffeomorphisms depending on $f$.

\subsection{Subgroups of volume-preserving flows on $\T^d$}

Recall that $\Z^d$ is an additive subgroup of $\R^d$. Let $\Gamma\subset\Z^d$ be a subgroup of $\Z^d$ such that $\mathrm{span}\Gamma=\R^d$. Let $f(x)=\sum\limits_{m\in\Z^d}p_me^{i\langle m,x\rangle}$ be the Fourier expansion of $f$, where $p_m\in \mathbb C^d,\ p_{-m}=\bar p_m$. We set 
$$
\M_f=\left\{m\in\Z^d \mid p_m\neq 0 \right\}, 
$$
$$
\mathrm{Vec}_0(\T^d)_\Gamma=\left\{f\in\mathrm{Vec}\T^d \mid \mathrm{div}f=0,\ \M_f\subset\Gamma\right\}.
$$
We omit the index $\Gamma$ if $\Gamma=\Z^d$.
It is easy to check that $\mathrm{Vec}_0(\T^d)_\Gamma$ is a closed Lie subalgebra of $\mathrm{Vec}\T^d$.

Now we consider the subgroup $\Gamma^*\subset\R^d$ dual to $\Gamma$,
$$
\Gamma^*=\left\{x\in\R^d \mid \langle x,y\rangle\in\Z,\ \forall\,y\in\Gamma\right\}.
$$
We see that $\Z^d\subset\Gamma^*$ and moreover, $\Gamma^*/\Z^d$ is a finite group. Indeed, $\Gamma=A\Z^d$, where $A$ is a nondegenerate 
matrix with integral entries. Then $\Gamma^*={A^*}^{-1}\Z^d$, where $A^*$ is the adjoint matrix of $A$, and ${A^*}^{-1}$ has rational entries.

Moreover, the group $2\pi\Gamma^*/2\pi\Z^d$ acts freely and properly on $\T^d$ by the translations and a divergence free vector field $f$ belongs to
$\mathrm{Vec}_0(\T^d)_\Gamma$ if and only if $f$ commutes with this action. The same property can be described in other way if we use
the covering $\mathfrak p_{\Gamma^*}:\T^d\to \T^d/2\pi\Gamma^*$. Here $\T^d/2\pi\Gamma^*$ is another torus. We see that 
$f\in\mathrm{Vec}_0(\T^d)_\Gamma$ if and only if $f=\mathfrak p^*_{\Gamma^*}g$ where $g\in\mathrm{Vec}_0(\T^d/2\pi\Gamma^*)$.

Any volume-preserving flow $P_t\in\mathrm{Diff}\T^d,\ P_0=\Id$, has a form $P_t=\E_0^tf_\tau\,d\tau$, where $f_\tau\in\mathrm{Vec}_0\T^d$.
This is true for any torus, in particular for the torus $\T^d/2\pi\Gamma^*$. We obtain that the flows generated by the time varying vector fields
from $\mathrm{Vec}_0(\T^d)_\Gamma$ are exactly the lifts to $\T^d$ of the volume-preserving flows on $\T^d/2\pi\Gamma^*$.

We denote by $\mathrm{Diff}_0\T^d_\Gamma$ the connected component of the identity in the group of volume-preserving diffeomorphisms of $\T^d$
commuting with the action of $\Gamma^*/\Z^d$. Then
$$
\mathrm{Diff}_0\T^d_\Gamma=\left\{\E_0^tf_\tau\,d\tau \mid f_\tau\in\mathrm{Vec}_0(\T^d)_\Gamma\right\}.
$$

\subsection{Approximation of volume-preserving diffeomorphisms by an affine in control system}

Recall that $\mathrm{Vec}_0\T^d$ is the set of divergence free vector fields of $\T^d$. We define the subset $\mathfrak V^d\subset\mathrm{Vec}_0\T^d$ as follows. A vector field $f\in\mathrm{Vec}_0\T^d$ belongs to $\mathfrak V^d$ if
\begin{enumerate}
\item[(i)] $\#\M_f<\infty$,   
\item[(ii)] $\mathrm{span}\M_f=\R^d$,
\item[(iii)] $\mathrm{span}\left\{f(x) \mid x\in\T^d\right\}=\R^d$.
\end{enumerate}

Clearly, $\overline{\mathfrak V^d}=\mathrm{Vec}_0\T^d$. Moreover, if $d=2$, then property (ii) implies (iii). Indeed if $f(x)=\sum_{m\in \M_f}a_m \cos \la m,x \ra+b_m \sin \la m,x\ra,x\in \T^d$, then $\mathrm{span}\left\{f(x)\mid x \in \T^d\right\}=\mathrm{span}\left\{a_m,b_m \mid m \in \M_f\right\}$. If $d=2$ and $\mathrm{span }\M_f=\R^2$, there exist $m,n \in \M_f$ such that $a_m,a_n\in \R^2\setminus \left\{0\right\}$ and $\la m,a_m\ra=\la n,a_n \ra=0$. Then necessarily $\mathrm{span}\left\{a_m,a_n\right\}=\R^2$.

Here we present the main results for system \eqref{affine system} with $f\in \mathfrak{V}^d$. The proofs are given in the following sections.

\begin{theorem} \label{theorem : Lie algebra T2 and T3}
Let $d=2$ or $d=3$ and $f\in\mathfrak V^d$. Let $\Gamma\subset\Z^d$ be the subgroup generated by $\M_f$. Then 
\begin{equation*}
    \overline{\mathrm{Lie}\{f+u \mid u\in\R^d\}}=\mathrm{Vec}_0(\T^d)_\Gamma.
\end{equation*}

\end{theorem}

\begin{theorem} \label{theorem : affine system approx diffeo}
Under the conditions of \cref{theorem : Lie algebra T2 and T3}, the subgroup $\mathrm{Diff}_0\T^d_\Gamma\subset\mathrm{Diff}_0\T^d$ is invariant
for system \eqref{affine system} and moreover the system is approximately controllable in $\mathrm{Diff}_0\T^d_\Gamma$.
\end{theorem}

\begin{theorem} \label{theorem : affine system ensemble controllability}
\begin{enumerate}
    \item[(i)] Let $d=2$ or $d=3$. There exists a residual subset $\mathcal R\subset\mathrm{Vec}_0\T^d$ such that, for every $f\in\mathcal R$ and $N\in \N^*$, system \eqref{affine system}
is globally controllable in the space of $N$-ensembles in $\T^d$.
\item[(ii)] For every $d\ge 2$, $f\in\mathrm{Vec}_0\T^d$, $N\ge 2$, and $T>0$, system \eqref{affine system} is not globally controllable 
for time smaller or equal than $T$ in the space of $N$-ensembles in $\T^d$.
\end{enumerate}   
\end{theorem}

\section{Proof of \cref{theorem : affine system approx diffeo}}

In what follows, $\mathrm{cone}S$ is the convex cone generated by the subset $S$ of a real vector space,
\begin{equation*}
    \mathrm{cone}S=\left\{\sum_i\alpha_ia_i \mid a_i\in S,\ \alpha_i\ge 0\right\},
\end{equation*} 
and $dw$ is the standard volume form on the torus. Let $f\in C^\infty(\T^d,\R^d)$ be a smooth vector function and $\theta\in\T^d$. We define a vector function $f_\theta$ by the formula
$f_\theta(x)=f(x+\theta),\ x\in\T^d$.
\\ \\
Let $t\geq 0$. By applying the variation formula to system \eqref{affine system}, see e.g. \cite[Section 2.7]{AgSa}, we obtain a decomposition of the flow
\begin{align*}
    P_t^{u(\cdot)}&=\f{\exp}\int_0^tf+u(\tau)d\tau \\
    &=\left(\f{\exp}\int_{0}^t \left(\Ad e^{\int_0^\tau u(s)ds}\right)fd\tau \right)\circ e^{\int_0^tu(s)ds} 
\end{align*}
where $(\Ad P^{-1})f=P_{*}f$ for any $P \in \mathrm{Diff} \T^d$. 
Let  $\theta(t)=\int_0^tu(s)ds$. Notice that $e^{\theta(t)}\in \mathrm{Diff} \T^d$ is the translation by $\theta(t)$. So
$$\left(\Ad e^{\int_0^\tau u(s)ds}\right)f=(\Ad e^{\theta(\tau)}) f=f_{\theta(\tau)}.$$Therefore, $$P_t^{u(\cdot)}=\left(\f{\exp}\int_0^t f_{\theta(\tau)}d\tau\right) \circ e^{\theta(t)}.$$
The map 
\begin{equation*}
    (\theta(\cdot),v) \mapsto \left(\f{\exp}\int_0^tf_{\theta(\tau)}d\tau \right)\circ e^v
\end{equation*}
is continuous from $L^1([0,t],\R^d)\times \R^d$ to $\mathrm{Diff}M$, and moreover the map 
\begin{equation}
    u(\cdot)\in L^1_{\mathrm{loc}}(\R_+,\R^d)\mapsto (\theta(\cdot),\theta(t))\in L^1([0,t],\R^d)\times \R^d
\end{equation}
has dense image in $L^1([0,t],\R^d)\times \R^d$, so the closure of the attainable set verifies 
\begin{equation*}
\left\{\f{\exp}\int_0^tf_{\theta(\tau)d\tau}\mid \theta(\cdot)\in L^1([0,t],\R^d)\right\}\circ\left\{e^v \mid v \in \R^d\right\}\subset \ov{\mathcal{A}}.
\end{equation*}

So the study of $ \ov{\mathcal{A}}$ is reduced to the study of the no more linear in the control system 
\begin{equation} \label{reduced system}
    \dot{x}=f_{\theta(t)}(x), \qquad \theta(t) \in \R^d,
\end{equation}
where $\int_{\T^d}f(\tau)d\tau=0$, see \cref{int0}, and $t\mapsto \theta(t)\in L^1_{\mathrm{loc}}(\R_+,\R^s)$. By standard convexification, see \cite[Th. 8.7]{AgSa}, the flow of system \eqref{reduced system} can approximate the flow of any convex combination of vector fields $f_\theta,\theta\in \R^d$. By re-scaling of the time, the flow of system \eqref{reduced system} can approximate the flow of any convex combination of vector fields $f_\theta, \theta \in \R^d$ up to a positive multiplicative constant, that is, the flow of any vector field in the convex subset 
$\mathrm{cone}\left\{f_\theta \mid \theta \in \R^d\right\}$.

\begin{lemma} \label{Lemma cone}
Let $f\in C^\infty(\T^d,\R^n)$. If $\int\limits_{\T^d}f(x)\,dw(x)=0$, then
$$
\overline{\mathrm{cone}\{f_\theta \mid \theta\in\T^d\}}=\overline{\mathrm{span}\{f_\theta \mid \theta\in\T^d\}},
$$
where the closure is taken in the $C^\infty$-topology.
\end{lemma}

\begin{proof}
    Assume that $\overline{\mathrm{cone}\{f_\theta\in\T^d\}}$ is not a vector space. Then, according to the standard separation theorem
for locally convex topological vector spaces, there exists $\varphi\in C^\infty(\T^d,\R^n)^*$ such that 
$\varphi$ restricted to ${\overline{\mathrm{span}\{f_\theta\mid \theta\in\T^d\}}}$ is not identically 0 and 
$\langle\varphi,f_\theta\rangle\le 0,\ \forall\,\theta\in\T^d$.

Note that $\theta\mapsto \langle \varphi,f_\theta\rangle$ is a continuous function on $\T^d$, hence it is strictly negative on an open subset of $\T^d$.
We have
$$
0>\int_{\T^d}\langle\varphi,f_\theta\rangle\,dw (\theta)=\langle\varphi,\int_{\T^d}f_\theta\,dw (\theta)\rangle.
$$
On the other hand, 
$$
\left(\int_{\T^d}f_\theta\,dw(\theta)\right)(x)=\int_{\T^d}f(x+\theta)\,dw(\theta)= \int_{\T^d}f(\theta)\,dw(\theta)=0,\qquad x\in\T^d.
$$
In other words, $\int\limits_{\T^d}f_\theta\,dw(\theta)=0$ and we obtain a contradiction which proofs the lemma.
\end{proof}

So to summarize, the flow of system \eqref{affine system} can approximate the flow of any vector field of the form $\alpha f_\theta + u$, with $\alpha \in \R$ and $\theta, u \in \R^d$. In particular, according to \cref{theorem : motion planning diffeo}, the flow of every vector field in $\ov{\mathrm{Lie}\left\{f_\theta +u \mid \theta,u \in \R^d\right\}}$ belongs to $\ov{\mathcal{A}}$. According to \cref{theorem : Lie algebra T2 and T3}, if $f \in \mathfrak{V}^d$ and if $\Gamma$ denotes the subgroup of $\Z^d$ generated by $\M_f$, then $\ov{\mathrm{Lie}\left\{f+u\mid u \in \R^d\right\}}=\mathrm{Vec}_0(\T^d)_\Gamma$ and $\ov{\mathcal{A}}=\mathrm{Diff}_0\T^d_\Gamma$.

\section{Proof of \cref{theorem : affine system ensemble controllability}} 

\textbf{(i)} Let us prove the first statement of \cref{theorem : affine system ensemble controllability}. Let $d=2$ or $d=3$. We recall that 
$$(\hat{\T}^d)^N = (\T^d)^N\setminus \left\{(y_1,\dots,y_N)\in (\T^d)^N \mid \exists k\neq \ell, y_k=y_\ell\right\}.$$
For every $f \in \mathrm{Vec}_0\T^d$, we consider the lift of control system \eqref{affine system} in the space of $N$-ensembles,
\begin{equation} \label{affine system ensemble controllability}
     \dot{x}_j=f(x_j)+u(t), \qquad u(t)\in \R^d,j \in \left\{1,\dots,N\right\},  
\end{equation}
where $x(t)=(x_1(t),\dots,x_N(t))\in (\hat{\T}^d)^N$ and $u(\cdot)$ is measurable and locally bounded. As explained in the proof of \cref{theorem : affine system approx diffeo}, the attainable set of system \eqref{affine system ensemble controllability} has the same closure of the attainable set of the following system,
\begin{equation} \label{linear system ensemble controllability}
    \dot{x}_j=\alpha(t)f(x_j)+u(t), \qquad \alpha(t)\in \R,u(t)\in \R^d,j \in \left\{1,\dots,N\right\},
\end{equation}
where $\alpha(\cdot)$ and $u(\cdot)$ are measurable and locally bounded. System \eqref{linear system ensemble controllability} is linear with respect to the control. According to the Rashevski -- Chow theorem, such a system is globally controllable if it is Lie bracket generating. Let us prove that there exists a residual set $\mathcal{R}\subset \mathrm{Vec}_0\T^d$ such that for every $f \in \mathcal{R}$, for every $N\in \N^*$, system \eqref{affine system ensemble controllability} is Lie bracket generating in $(\hat{\T}^d)^N$.

Let us fix $N \in \N^*$. We consider $f\in \mathfrak{V}^d$ such that $\Gamma =\Z^d$, then according to \cref{theorem : Lie algebra T2 and T3}, $\ov{\mathrm{Lie}\left\{f+u \mid u \in \R^d\right\}}=\mathrm{Vec}_0\T^d$. In this case, as explained in the proof of \cref{theorem : controllability ensembles}, system \eqref{affine system ensemble controllability} is Lie bracket generating at every point of $(\hat\T^d)^N$. 
\begin{rem}
    Although the set $\left\{f \in \mathfrak{V}^d \mid \Gamma = \Z^d\right\}$ is dense in $\mathrm{Vec}_0\T^d$, it is not residual.
\end{rem}
The manifold $(\hat\T^d)^N$ is the union of a countable number of compacts, $(\hat\T^d)^N=\bigcup\limits_iK_{Ni}$, where 
$K_{Ni}\Subset(\hat\T^d)^N,\ i=1,2,\dots$. The set of vector fields $f\in\mathrm{Vec}_0\T^d$ such that system \eqref{affine system ensemble controllability} is Lie bracket generating at
every point of $K_{Ni}$ is open. Moreover, we know that it is dense, hence it is open dense. The desired residual set is just the intersection
of these open dense subsets for all $Ni$.
\\ \\
\textbf{(ii)} Let us prove the second statement of \cref{theorem : affine system ensemble controllability}. Let $d\geq 2$, $f\in \mathrm{Vec}_0\T^d$, $N\geq 2$ and $T>0$. Let $t\mapsto x(t)=(x_1(t),\dots,x_N(t))$ be the solution of \eqref{affine system ensemble controllability}. For every $t \in [0,T]$, $x_1(t)\neq x_2(t)$ because $x_1(0)\neq x_2(0)$, and so $t\mapsto \xi(t)=\ln|x_1(t)-x_2(t)|$, $t\in [0,T]$, is well defined. Moreover, for every $t \in [0,T]$,
\begin{equation*}
    \dot{\xi}(t)=\frac{\la f(x_1(t))-f(x_2(t)),x_1(t)-x_2(t)\ra}{|x_1(t)-x_2(t)|} \\
    \geq -\|f\|_1,
\end{equation*}
and so $|\xi(t)|\leq \xi(0)-T\|f\|_1$. Then $|x_1(t)-x_2(t)|\geq e^{-T\|f\|_1}|x_1(0)-x_2(0)|$ for every $t \in [0,T]$, and so the configurations where $x_1(t)$ and $x_2(t)$ are very close are not reachable in any time $t \in [0,T]$.

\section{Proof of \cref{theorem : Lie algebra T2 and T3}} \label{section : Lie algebra computation T2 T3}

The proof of \cref{theorem : Lie algebra T2 and T3} requires several steps and the study depends on the dimension of the considered torus. For the bi-dimensional torus, the statement of \cref{theorem : Lie algebra T2 and T3} is proved by \cref{theorem : algebra dim2}, and for the tri-dimensional torus by \cref{theorem : algebra dim3 span3}. 

\subsection{Bi-dimensional torus}

On $\T^2$, the volume form $dx \we dy$ coincides with the symplectic form, and every divergence free vector field  can be written as the sum of a Hamiltonian vector field $\f{h} \in \mathrm{Ham}\T^2$ and a constant vector field. Indeed, if we denote $\omega=dx \we dy$, according to Cartan's formula, the Lie derivative of $\omega$ along any vector field $V\in \mathrm{Vec}\T^2$ verifies
    \begin{equation*}
        \mathcal{L}_V\omega= (i_V\circ d + d \circ i_V)\omega=d\circ i_V\omega.
    \end{equation*}
    If $\mathrm{div}V=0$, then $i_V\omega$ is closed, so there exists a constant vector field $u=u_1\p_{x_1}+u_2\p_{x_2}$ such that $di_{V+u}\omega=0$, so $V+u$ verifies $\mathcal{L}_{V+u}\omega=0$ and $V+u$ is Hamiltonian.
    
For this reason we can assume that there exists a smooth function $h \in C^{\infty}(\T^2,\R)$, associated to the Hamiltonian vector field 
\begin{equation*}
    \f{h}(x,y)=-\frac{\p h}{\p y}(x,y)\p_x + \frac{\p h}{\p x}(x,y)\p_y, \qquad (x,y)\in \T^2,
\end{equation*}
such that $f=\f{h}$.

The non-zero modes that appear in the Fourier decomposition of the function $h$ are exactly those that appear in the decomposition of the vector field $\f{h}$. The set of modes in the decomposition of $h$ is denoted by $\M_h$, and the subgroup of $\Z^2$ generated by $\M_h$ is denoted by $\Gamma$. Note that the subgroups of $\Z^2$ generated by $\M_h$ and $\M_f$ are the same.  We recall that for $a,b \in C^{\infty}(\T^2,\R)$, their Poisson bracket is defined by 
\begin{equation*}
    \left\{a,b\right\}=\frac{\p a}{\p x}\frac{\p b}{\p y}- \frac{\p a}{\p y}\frac{\p b}{\p x},
\end{equation*}
and the arrow map $C^{\infty}(\T^2,\R)\mapsto \mathrm{Ham}\T^2$ preserves the Lie algebra structure owing to the relation 
\begin{equation*}
    \f{ \left\{ a,b \right\}}=\left[ \f{a},\f{b} \right].
\end{equation*}
\begin{theorem} \label{theorem : algebra dim2} 
Let $\# \M_h < \infty$. 

\begin{itemize}
    \item If $\mathrm{span} \Gamma=\R^2$, then 
    \begin{equation*}
        \mathrm{Lie}\left\{\f{h}+u\mid u \in \R^d\right\} = \mathrm{span} \left\{ \f{\cos}\la m, \cdot \ra, \f{\sin}\la m, \cdot \ra,\p_x,\p_y \mid m \in \Gamma  \right\}.
    \end{equation*}
    \item If $\mathrm{span} \Gamma$ is of dimension 1, then
    \begin{equation*}
        \mathrm{Lie}\left\{\f{h}+u\mid u \in \R^d\right\} = \mathrm{Lie} \left\{ \f{\cos}\la m, \cdot \ra, \f{\sin}\la m, \cdot \ra,\p_x,\p_y \mid m \in \M_h  \right\}.
    \end{equation*}
\end{itemize}
\end{theorem}

Throughout the proof of \cref{theorem : algebra dim2}, we will use the notation $\mathfrak L_h = \mathrm{Lie}\left\{\f{h}+u \mid u \in \R^d\right\}$. We will also make use of the following identity: 
\begin{equation}\label{eqn:lemma derivation}
        \ad_{\p_x}^k\ad_{\p_y}^\ell\f{h}=\f{\frac{\p^{k+\ell}h}{\p x^k\p y^\ell}} \in \mathfrak{L}_h,\qquad  k,\ell\in \N.
\end{equation}
    
\begin{lemma} \label{lemma canceled modes dim2} If $\# \M_h < \infty$, then 
\begin{equation*}
        \mathfrak{L}_h=\mathrm{Lie} \left\{ \f{\cos}\la m,  \cdot \ra,\p_x,\p_y \mid m \in \M_h\right\}.
\end{equation*}
\end{lemma}

\begin{proof} Let $h=\sum_{m\in \M_h}h_me^{i \la m, \cdot \ra}$ be the finite Fourier decomposition of $h$, where the coefficients $h_m$ are complex. The function $f$ is real-valued so $h_{-m}=\ov{h_m}$ for every $m\in \M_h$. Let us prove that $\f{\cos}\la m_0, \cdot \ra \in \mathfrak{L}_h$ for every $m_0 \in \M_h$. By a straightforward computation, 
\begin{equation*}
    \ad^2_{\p_x}\f{h}=-\sum_{m \in \M_h}m_x^2h_m\f{e^{i \la m, \cdot\ra}} \in \mathfrak{L}_h,
    \end{equation*}
and then for every $\alpha,\beta \in \R$,
\begin{equation*}
    (\alpha - \ad_{\p_x}^2)(\beta - \ad_{\p_y}^2)\f{h}=\sum_{m\in \M_h}(\alpha-m_x^2)(\beta - m_y^2)h_me\f{^{i \la m, \cdot\ra}} \in \mathfrak{L}_h.
\end{equation*}
Let $m_0=(m_{0x},m_{0y})\in \M_h$. For any $m\in \Z^2$ we denote $|m|=|m_0|$ if $|m_x|=|m_{0x}|$ and $|m_y|=|m_{0y}|$. By iteration and thanks to a specific choice of $\alpha,\beta\in \R$, we obtain that
\begin{equation*}
    \prod_{m_1,m_2\in \M_h \atop m_{1x}\neq m_{0x}, m_{2y}\neq m_{0y}}(m_{1x}^2-\ad_{\p_x}^2)(m_{2y}^2-\ad_{\p_y}^2)\f{h}=\gamma \sum_{|m|=|m_0|}h_m\f{e^{i \la m , \cdot\ra }} \in \mathfrak{L}_h,
\end{equation*}
where 
\begin{equation*}
    \gamma = \prod_{m_1,m_2\in \M_h \atop m_{1x}\neq m_{0x}, m_{2y}\neq m_{0y}}(m_{1x}^2-m_{0x}^2)(m_{2y}^2-m_{0y}^2)\neq 0,
\end{equation*}
and so $\sum_{|m|=|m_0|}h_m\f{e^{i \la m , \cdot\ra }} \in \mathfrak{L}_h$. The function $h$ is real-valued, so $\ov{h_m} =h_{-m}$ for every $m\in \M_h$. If $m_0=(m_{0x},m_{0y})$, we denote $m_0'=(m_{0x},-m_{0y})$, and then
\begin{equation*}
    \sum_{|m|=|m_0|}h_m\f{e^{i \la m , \cdot\ra }}=2\mathfrak{Re}(h_{m_0}\f{e^{i\la m_0, \cdot \ra}} + h_{m_0'}\f{e^{i\la m_0', \cdot \ra}}) \in \mathfrak{L}_h. 
\end{equation*}
Let us consider the case where $m_{0x},m_{0y}\neq 0$. According to formula \eqref{eqn:lemma derivation},
\begin{equation*}
    -\sum_{|m|=|m_0|}h_m\f{\frac{\p^2}{\p x\p y}e^{i \la m , \cdot\ra }}=2m_{0x}m_{0y}\mathfrak{Re}(h_{m_0}\f{e^{i\la m_0, \cdot \ra}} - h_{m_0'}\f{e^{i\la m_0', \cdot \ra}}) \in \mathfrak{L}_h,
\end{equation*}
so by linear combination $\mathfrak{Re}(h_{m_0}\f{e^{i\la m_0, \cdot \ra}})=\mathfrak{Re}(h_{m_0})\f{\cos} \la m_0,\cdot \ra - \mathfrak{Im}(h_{m_0})\f{\sin}\la m_0,\cdot\ra\in \mathfrak{L}_h$. Taking the derivative with respect to one variable we obtain that
\begin{equation*}
    -\mathfrak{Re}(h_{m_0})\f{\sin} \la m_0,\cdot \ra - \mathfrak{Im}(h_{m_0})\f{\cos}\la m_0,\cdot\ra \in \mathfrak{L}_h,
\end{equation*}
and so by linear combination, $(\mathfrak{Re}(h_{m_0})^2+\mathfrak{Im}(h_{m_0})^2)\f{\cos}\la m_0, \cdot \ra \in \mathfrak{L}_h$, and so $\f{\cos}\la m_0,\cdot \ra \in \mathfrak{L}_h$. The other cases can be easily derived from the previous one.
\end{proof}

\begin{lemma} \label{addition modes dim 2}
    Let $m=(m_1,m_2)$ and $n=(n_1,n_2)$. Let $m \we n = m_1n_2-m_2n_1$. If $m,n \in \M_h$ and if $m\we n \neq 0$, then $\f{\cos}\la m+n, \cdot \ra \in \mathfrak{L}_h$.
\end{lemma}

\begin{proof}
According to \cref{lemma canceled modes dim2}, 
\begin{align*}
    \left\{\sin\la n, \cdot \ra,\cos\la m, \cdot \ra \right\}&=(m \we n)\sin \la m, \cdot \ra \cos \la n, \cdot \ra \in \mathfrak{L}_h, \\
    \left\{\cos\la n, \cdot \ra,\sin \la m, \cdot \ra \right\}&=(m \we n) \cos \la m,\cdot \ra \sin \la n, \cdot \ra \in \mathfrak{L}_h.
\end{align*}
So by linear combination $\f{\sin} \la m+n, \cdot \ra \in \mathfrak{L}_h$. By \eqref{eqn:lemma derivation}, $\f{\cos}\la m+n, \cdot \ra \in \mathfrak{L}_h$.
\end{proof}

\begin{proof}[Proof of \cref{theorem : algebra dim2}]
If $\mathrm{span} \M_h$ is of dimension 1, we can assume up to an orthonormal change of variables that $\frac{\p}{\p y}h=0$. The Poisson Bracket of two functions that only depend on $x$ is zero, so 
\begin{equation*}
    \mathrm{Lie}\left\{ \f{\cos}\la m,\cdot \ra,\p_x,\p_y \mid m \in \M_h\right\}=\mathrm{span} \left\{\f{\cos}\la m,\cdot \ra, \f{\sin}\la m,\cdot \ra,\p_x, \p_y \mid m \in \mathfrak{L}_h \right\}.
\end{equation*}
If $\mathrm{span}\M_h=\R^2$, let us introduce the sets $\mathcal{I}_k(h),k \in \N^{*}$, defined by $\mathcal{I}_0(h)=\M_h$ and 
\begin{equation*}
    \mathcal{I}_{k+1}(h)=\mathcal{I}_k(h) \cup \left\{m+n \: | \: m,n \in \mathcal{I}_k(h), m \we n \neq 0\right\}.
\end{equation*}
According to \cref{addition modes dim 2}, $\f{\cos}\la m, \cdot \ra \in \mathfrak{L}_h$ for every $m \in \cup_{k \in \N}\mathcal{I}_k(h)$. But if $\mathrm{span} \M_h=\R^2$, it is easy to see that $\cup_{k \in \N}\mathcal{I}_k(h)=\Gamma$. Indeed every element $m \in \Gamma$ can be written as a sum $m=m_1\pm\dots \pm m_p$, with $m_1,\dots, m_p\in \M_h$. Note that if $m \in \M_h$ then it is also verified that $\f{\cos}\la -m, \cdot \ra,\f{\sin}\la -m,\cdot \ra \in \mathfrak{L}_h$. If $m=m_1 +m_2$ and if $m_1 \we m_2=0$, necessarily there exists $m_3\in \M_h$ such that $m_1 \we m_3 \neq 0$. Then $m_1+m_3\in \cup_{k \in \N}\mathcal{I}_k(h)$ and $(m_1 + m_3)\we m_2 \neq 0$, so $m_1+m_2+m_3 \in \cup_{k \in \N}\mathcal{I}_k(h)$, and $(m_1+m_2+m_3)\neq -m_3$, so finally $m_1+m_2 \in \cup_{k \in \N}\mathcal{I}_k(h)$ and 
\begin{equation*}
    \mathfrak{L}_h=\mathrm{Lie} \left\{ \f{\cos}\la m, \cdot \ra, \p_x,\p_y \mid m \in \Gamma\right\}=\mathrm{span} \left\{\f{\cos}\la m, \cdot \ra, \f{\sin}\la m, \cdot \ra,\p_x, \p_y \mid m \in \Gamma \right\}.
\end{equation*}
Indeed the Lie algebra is composed of linear combinations and derivatives of the modes present in $\Gamma$, which is closed. 
\end{proof}

In order to check that $\Gamma=\Z^2$, and so that $\mathrm{Lie}\left\{\f{h}+u \mid u \in \R^d\right\}$ is dense in $\mathrm{Vec}_0M$, we can apply the following criterion from \cite[Lem. 1]{AgSa05}.

\begin{lemma}
    The subgroup generated by $\M_h$ is equal to $\Z^2$ if and only if the greatest common divisor (g.c.d) of the numbers $\left\{ m \we n \mid m,n \in \M_h\right\}$ equals 1.
\end{lemma}

\subsection{Tri-dimensional torus}

On $\T^3$, we use the Fourier decomposition of a divergence free vector field, 
$$f=\sum_{m \in \M_f}p_me^{i\la m,\cdot \ra}=\sum_{m\in \M_f}a_m\cos\la m,\cdot \ra +b_m\sin\la m, \cdot \ra,$$ 
where $p_m$ are linear combinations of $\p_x,\p_y,\p_z$. We identify the constant vector fields $p_m$ with vectors in $\C^3$, whose coordinates correspond to the coefficients in $\p_x,\p_y,\p_z$. In particular $p_m=\frac{a_m-ib_m}{2}$. The components of the vector $\mathfrak{Re}(p_m)$ (respectively $\mathfrak{Im}(p_m)$) correspond to the real parts (respectively to the imaginary parts) of the components of $p_m$. With these notations, and because $\mathrm{div} f=0$, $\la m,p_m\ra =\la m,a_m\ra=\la m,b_m \ra=0$ for every $m \in \M_f$. The set of directions that are orthogonal to $m\in \Z^3$ is denoted by $m^{\perp}:=\left\{v \in \R^3\mid \la m,v \ra =0\right\}$. 
For two vectors $a,b\in \R^3$, their cross product is denoted by $a\we b$. The subgroup of $\Z^3$ generated by $\M_f$ is denoted by $\Gamma$. The aim of this section is to characterize the Lie algebra $\mathrm{Lie}\left\{f+u \mid u \in \R^3\right\}$. We will use the notation $\mathfrak{L}_f=\mathrm{Lie}\left\{f+u \mid u \in \R^3\right\}$. 
\\ \\
In the following, we will make use of the following formulas. For every $m \neq 0$,
\begin{align*}
    &\ad_{\p_x}a_m\cos\la m,\cdot \ra=-m_xa_m\sin \la m,\cdot \ra, \quad \ad_{\p_x}b_m\sin\la m,\cdot\ra=m_xb_m\cos \la m,\cdot \ra,\\
    &\ad_{\p_y}a_m\cos\la m,\cdot \ra=-m_ya_m\sin \la m,\cdot \ra,\quad \ad_{\p_y}b_m\sin\la m,\cdot\ra=m_yb_m\cos \la m,\cdot \ra,\\
    &\ad_{\p_z}a_m\cos\la m,\cdot \ra=-m_za_m\sin \la m,\cdot \ra, \quad \ad_{\p_z}b_m\sin\la m,\cdot\ra=m_zb_m\cos \la m,\cdot \ra. 
\end{align*}

\begin{lemma} \label{lemma canceled modes dim3}

Let $\# \M_f < \infty$. Then
    \begin{equation*}
        \mathfrak{L}_f=\mathrm{Lie}\left\{a_m\cos \la m,\cdot \ra + b_m\sin \la m,\cdot \ra,\p_x,\p_y,\p_z \mid m \in \M_f\right\}.
    \end{equation*}
\end{lemma}

\begin{proof}
     As for the bi-dimensional case, we explain how the isolated frequencies also belong to the Lie algebra. Indeed, let $m_0\in \M_f$, let us prove that 
     \begin{equation*}
         a_{m_0}\cos \la m_0,\cdot \ra +b_{m_0}\sin \la m_0,\cdot \ra \in \mathfrak{L}_f.
     \end{equation*}
     By a straightforward computation,
    \begin{equation*}
        \ad_{\p_x}^2 f=-\sum_{m\in \M_f}m_x^2p_me^{i\la m,\cdot \ra}\in \mathfrak{L}_f,
    \end{equation*}
    and so for every $\alpha \in \R$,
    \begin{equation*}
        \alpha f - \ad_{\p_x}^2 f =\sum_{m \in \M_f}(\alpha-m_x^2)p_me^{i\la m,\cdot \ra}\in \mathfrak{L}_f.
    \end{equation*}
    If there exists $m_1 \in \M_f$ such that $|m_{1x}|\neq |m_{0x}|$, then 
    \begin{equation*}
        (m_{1x}^2-\ad_{\p_x}^2)f=\sum_{m \in \M_f}(m_{1x}^2-m_x^2)p_me^{i \la m,\cdot \ra}\in \mathfrak{L}_f.
    \end{equation*}
    By iteration of such operation for every $m_1 \in \M_f$ that verifies $|m_{1x}|\neq |m_{0x}|$, we obtain that
    \begin{equation*}
        \prod_{m_1\in \M_f \atop |m_{1x}|\neq |m_{0x}|}(m_{1x}^2-\ad_{\p_x}^2)f= \beta\sum_{m\in \M_f\atop|m_x|=|m_{0x}|}p_me^{i \la m,\cdot \ra}\in \mathfrak{L}_f,
    \end{equation*}
    where 
    \begin{equation*}
        \beta = \prod_{m_1\in \M_f\atop |m_{1x}|\neq |m_{0x}|}(m_{1x}^2-m_{0x}^2) \neq 0.
    \end{equation*}
    For any $m\in \Z^3$, $|m|=|m_0|$ means that $|m_x|=|m_{0x}|,|m_y|=|m_{0y}|$ and $|m_z|=|m_{0z}|$. By iteration and thanks to an adapted choice of $\alpha,\beta$, we obtain that
    \begin{equation*}
    \prod_{m_1,m_2,m_3\in \M_f \atop |m_{1x}|\neq |m_{0x}|,|m_{2y}|\neq |m_{0y}|,|m_{3z}|\neq |m_{0z}|}(m_{1x}^2-\ad_{\p_x}^2)(m_{2y}^2-\ad_{\p_y}^2)(m_{3z}^2-\ad_{\p_z}^2)f= \gamma \sum_{m\in \M_f\atop |m|=|m_0|}p_me^{i\la m,\cdot \ra}\in \mathfrak{L}_f,
    \end{equation*}
    where 
    \begin{equation*}
        \gamma =\prod_{m_1,m_2,m_3\in \M_f \atop |m_{1x}|\neq |m_{0x}|,|m_{2y}|\neq |m_{0y}|,|m_{3z}|\neq |m_{0z}|}(m_{1x}^2-m_{0x}^2)(m_{2y}^2-m_{0y}^2)(m_{3z}^2-m_{0z}^2)\neq 0,
    \end{equation*}
    and so 
    \begin{equation*}
        \sum_{|m|=|m_0|}p_me^{i\la m,\cdot \ra}\in \mathfrak{L}_f.
    \end{equation*} Let us consider the case where $m_{0x},m_{0y},m_{0z}\neq 0$. There are $2^{3}=8$ modes $m\in \Z^3$ that verify $|m|=|m_0|$. The vector field $f$ is real-valued, so $\ov{p_m}=p_{-m}$ for every $m\in \M_f$. There are $4$ couples of opposite modes $m\in\Z^3$ that verify $|m|=|m_0|$, so 
    \begin{equation*}
        \sum_{|m|=|m_0|}p_me^{i\la m,\cdot \ra}=2(\mathfrak{Re}(p_{m_0}e^{i\la m_0,\cdot\ra})+\sum_{k=1}^3\mathfrak{Re}(p_{m_{0,k}}e^{i\la m_{0,k},\cdot\ra}))\in \mathfrak{L}_f,
    \end{equation*}
    where 
    \begin{equation*}
       m_{0,1}=(m_{0x},m_{0y},-m_{0z}),\quad m_{0,2}=(m_{0x},-m_{0y},m_{0z}), \quad m_{0,3}=(-m_{0x},m_{0y},m_{0z}).
    \end{equation*}
     Then  
    \begin{equation*}
         \frac{1}{m_{0x}}(m_{0x}+\ad_{\p_x})\sum_{|m|=|m_0|}p_me^{i\la m,\cdot \ra}=4(\mathfrak{Re}(p_{m_0}e^{i\la m_0,\cdot\ra})+\sum_{k=1}^2\mathfrak{Re}(p_{m_{0,k}}e^{i\la m_{0,k},\cdot \ra}))\in \mathfrak{L}_f.
    \end{equation*}
    Then we can apply $\frac{1}{m_{0y}}(m_{0y}+\ad_{\p_y})$ and $\frac{1}{m_{0z}}(m_{0z}+\ad_{\p_z})$ to the previous vector field and we obtain that
    \begin{equation*}
        \mathfrak{Re}(p_{m_0} e^{i\la {m_0},\cdot \ra})=a_{m_0}\cos\la {m_0},\cdot \ra +b_{m_0}\sin \la {m_0},\cdot \ra\in \mathfrak{L}_f.
    \end{equation*}
    The other cases can be easily derived from the previous one.
\end{proof}

The following formulas can be obtained by straightforward computations and will be useful for the remaining proofs. 

\begin{prop} \label{formulas T3}
\begin{align*}
    1)\quad &[p_m\sin\la m,\cdot \ra,p_n\cos \la n,\cdot \ra]=\la m,p_n\ra p_m\cos \la m,\cdot \ra \cos \la n,\cdot\ra +\la n,p_m\ra p_n\sin \la m,\cdot \ra \sin \la n,\cdot \ra,  \\
    &[p_m\cos \la m,\cdot \ra,p_n \sin \la m,\cdot \ra]=-\la m,p_n\ra p_m\sin \la m,\cdot \ra \sin \la n,\cdot\ra -\la n,p_m\ra p_n\cos \la m,\cdot \ra \cos \la n,\cdot \ra, \\
    &[p_m\cos \la m,\cdot \ra,p_n \cos \la m,\cdot \ra]=-\la m,p_n\ra p_m\sin \la m,\cdot \ra \cos \la n,\cdot\ra +\la n,p_m\ra p_n\cos \la m,\cdot \ra \sin \la n,\cdot \ra, \\
    &[p_m\sin \la m,\cdot \ra,p_n \sin \la m,\cdot \ra]=-\la m,p_n\ra p_m\cos \la m,\cdot \ra \sin \la n,\cdot\ra -\la n,p_m\ra p_n\sin \la m,\cdot \ra \cos \la n,\cdot \ra. \\ \\
    2)\quad &[p_m\sin \la m,\cdot \ra,p_n\sin\la n,\cdot\ra ]-[p_m\cos \la m,\cdot \ra,p_n\cos \la n,\cdot \ra ]
    =(\la m,p_n\ra p_m - \la n,p_m\ra p_n)\sin \la m+n,\cdot \ra,\\
    &[p_m\cos \la m,\cdot \ra,p_n\sin\la n,\cdot\ra ]+[p_m\sin \la m,\cdot \ra,p_n\cos \la n,\cdot \ra ]
    =(\la m,p_n\ra p_m - \la n,p_m\ra p_n)\cos \la m+n,\cdot \ra.\\ \\
    3) \quad &[a_m\cos\la m,\cdot \ra+b_m\sin \la m,\cdot \ra,c_n\cos \la n,\cdot \ra+d_n \sin \la n,\cdot \ra]\\
        -&[-a_m\sin\la m,\cdot \ra+b_m\cos \la m,\cdot \ra,-c_n\sin \la n,\cdot \ra+d_n \cos \la n,\cdot \ra]\\
        &=(\la m,c_n\ra b_m+\la m,d_n\ra a_m- \la n,b_m \ra c_n - \la n,a_m \ra d_n)\cos \la m+n,\cdot \ra \\
        +&(\la m,d_n\ra b_m-\la m,c_n \ra a_m -\la n,b_m \ra d_n+\la n,a_m \ra c_n)\sin \la m+n,\cdot \ra.
    \end{align*}
\end{prop}

\begin{theorem} \label{theorem : appendix T3 span1}
     Let $\# \M_f < \infty$. If $\mathrm{span}\M_f$ is of dimension 1 or if ($\mathrm{span}\M_f$ is of dimension 2 and $\mathrm{span}\left\{a_m,b_m \mid m \in \M_f\right\}$ is of dimension 1), then 
        \begin{equation*}
            \mathrm{Lie}\left\{f+u \mid u \in \R^3\right\}=\mathrm{span}\left\{(\alpha a_m + \beta b_m)\cos \la m,\cdot \ra + (-\beta a_m + \alpha b_m)\sin \la m,\cdot \ra,\p_x,\p_y,\p_z \mid m\in \M_f, \alpha,\beta \in \R\right\},
        \end{equation*}
\end{theorem}

\begin{proof} Let $m\in \M_f$ be such that $m\neq 0$. According to \cref{lemma canceled modes dim3}, $a_m\cos \la m,\cdot \ra+b_m\sin\la m,\cdot \ra \in \mathfrak{L}_f$. If $m_x \neq 0$, 
    \begin{equation*}
        \ad_{\p_x}(a_m\cos \la m,\cdot \ra+b_m\sin\la m,\cdot \ra)=\underbrace{m_x}_{\neq0}(-a_m\sin\la m,\cdot \ra+b_m\cos \la m,\cdot \ra) \in \mathfrak{L}_f.
    \end{equation*}
    Else we compute $\ad_{\p_y}$ or $\ad_{\p_z}$, and because $m\neq0$, 
    \begin{equation*}
        -a_m\sin\la m,\cdot \ra+b_m\cos \la m,\cdot \ra \in \mathfrak{L}_f,
    \end{equation*}
    so for every $\alpha,\beta \in \R$,
    \begin{equation*}
        (\alpha a_m+\beta b_m)\cos \la m,\cdot \ra+(\alpha b_m - \beta a_m)\sin \la m,\cdot \ra \in \mathfrak{L}_f,
    \end{equation*}
    and 
    \begin{align*}
        &\mathrm{Lie} \left\{a_m\cos \la m,\cdot \ra+b_m\sin \la m,\cdot \ra,\p_x,\p_y,\p_z\right\} \\
        =&\mathrm{span} \left\{(\alpha a_m+\beta b_m)\cos \la m,\cdot \ra+(\alpha b_m -\beta a_m)\sin \la m,\cdot \ra,\p_x,\p_y,\p_z \mid \alpha,\beta \in \R \right\}.
    \end{align*}
    Let $n\in \M_f$ be another mode of $f$, then $a_n\cos \la n,\cdot \ra + b_n \cos \la n,\cdot \ra \in \mathfrak{L}_f$ according to \cref{lemma canceled modes dim3}. If $\mathrm{span}\M_f$ is of dimension 1, then $m\we n=0$ and $m^{\perp}=n^{\perp}$. If $\mathrm{span}\M_f$ is of dimension 2 and if $m\we n\neq0$, then $\mathrm{span}\M_f=\mathrm{span} \left\{m,n\right\}$. But $\mathrm{span}\left\{a_m,b_m \mid m\in \M_f\right\}$ is of dimension 1 and $\la k,a_k\ra=0$ for every $k\in \M_f$, so necessarily $\mathrm{span}\left\{a_m,b_m\mid m \in \M_f\right\}=\mathrm{span} \left\{m\we n\right\}$. In each case, $a_m,b_m,a_n,b_n \in \mathrm{span} \left\{m\we n\right\}$, and
    \begin{align*}
        \la m,a_n\ra=\la n,a_m\ra=\la m,b_n\ra=\la n,b_m\ra=0.
    \end{align*}
    Applying the third formula of \cref{formulas T3},
    \begin{equation*}
        [a_m\cos\la m,\cdot \ra +b_m\sin \la m,\cdot \ra,a_n\cos\la n,\cdot\ra +b_n \sin \la n,\cdot \ra]=0,
    \end{equation*}
    and 
    \begin{equation*}
        \mathfrak{L}_f=\mathrm{span} \left\{(\alpha a_m+\beta b_m)\cos \la m,\cdot \ra+(\alpha b_m-\beta a_m)\sin \la m,\cdot \ra,\p_x,\p_y,\p_z \mid m\in \M_f, \alpha,\beta \in \R\right\}. 
    \end{equation*}
\end{proof}

    \begin{rem} \label{T3 rem equivalences}
    If $m\in \M_f$, then 
    \begin{equation*}
        \forall p_m \in m^{\perp}, p_m \cos \la m,\cdot \ra \in \mathfrak{L}_f \iff \forall p_m \in m^{\perp}, p_m\sin \la m,\cdot \ra \in \mathfrak{L}_f,
    \end{equation*}
    and 
    \begin{equation*}
        (\forall p_m,q_m\in m^{\perp}, \quad p_m\cos\la m,\cdot \ra +q_m \sin \la m,\cdot \ra\in \mathfrak{L}_f) \iff(\forall p_m\in m^{\perp}, \quad p_m\cos \la m,\cdot \ra\in \mathfrak{L}_f  ). 
    \end{equation*}
\end{rem}

\begin{theorem} \label{theorem : algebra dim3 span3}
    Let $\# \M_f < \infty$. If $\mathrm{span}\M_f=\R^3$ and $\mathrm{span}\left\{a_m,b_m \mid m \in \M_f\right\}=\R^3$,
    then  
    \begin{equation*}
        \mathrm{Lie}\left\{f+u \mid u \in \R^3\right\}=\mathrm{span} \left\{p_m\cos\la m,\cdot \ra,q_m\sin\la m,\cdot\ra,\p_x,\p_y,\p_z \mid m\in \Gamma,p_m,q_m \in m^{\perp}\right\}.
    \end{equation*}
\end{theorem}

Before going to the proof of \cref{theorem : algebra dim3 span3}, we need to introduce several propositions.

\begin{prop} \label{adding modes : different angles}
    Let $m\in \M_f$ be such that $m\neq 0$. If there exists $a_m,b_m,a_m',b_m'$ constant vector fields such that $a_m,a_m'\neq 0$,
    \begin{equation*}
        \left\{\begin{array}{cc}
           a_m\cos \la m,\cdot\ra+b_m \sin \la m,\cdot \ra \in \mathfrak{L}_f,    \\
        a_m'\cos\la m,\cdot \ra+b_m'\sin \la m,\cdot \ra\in \mathfrak{L}_f,
        \end{array}\right.
    \end{equation*}
    and such that $\begin{pmatrix}
            a'_m \\
            b'_m
        \end{pmatrix}\neq \begin{pmatrix}
            \alpha & \beta\\
            -\beta & \alpha 
        \end{pmatrix}\begin{pmatrix}
            a_m\\
            b_m
        \end{pmatrix}$ for every $\alpha,\beta\in \R$, then for every $p_m,q_m\in m^{\perp}$,
    \begin{equation*}
        p_m\cos \la m,\cdot \ra + q_m \sin \la m,\cdot \ra \in \mathfrak{L}_f.
    \end{equation*}
\end{prop}

\begin{proof}
    According to \cref{T3 rem equivalences}, this is sufficient to show that there exist two non-colinear vectors $p_m,q_m\in m^{\perp}$, such that  \begin{align*}
        (p_m\cos \la m,\cdot \ra\in \mathfrak{L}_f \quad \text{or} \quad p_m\sin \la m,\cdot \ra\in \mathfrak{L}_f), \\
        \text{and}\quad (q_m\cos \la m,\cdot \ra \in \mathfrak{L}_f \quad \text{or} \quad q_m\sin \la m,\cdot \ra \in \mathfrak{L}_f).
    \end{align*} 
    According to \cref{theorem : appendix T3 span1}, for every $\alpha,\beta\in \R$, 
    \begin{equation} \label{T3 prop1 rotation formulas}
    \left\{\begin{array}{ll}
         (\alpha a_m+\beta b_m)\cos \la m,\cdot \ra + (\alpha b_m -\beta a_m)\sin \la m,\cdot \ra \in \mathfrak{L}_f,  \\
         (\alpha' a_m'+\beta' b_m')\cos \la m,\cdot \ra + (\alpha' b_m' -\beta' a_m')\sin \la m,\cdot \ra \in \mathfrak{L}_f.
    \end{array}\right.
    \end{equation}
    \underline{If $a_m' \we b_m'=0$ and $a_m\we b_m=0$:} Thanks to an adapted choice of $\alpha,\beta,\alpha',\beta'$, we obtain that 
    \begin{align*}
        a_m\cos \la m,\cdot \ra\in \mathfrak{L}_f \quad \text{and} \quad 
        a_m'\cos \la m,\cdot \ra\in \mathfrak{L}_f,
    \end{align*}
    and necessarily $a_m\we a_m'\neq 0$, otherwise it would exist $\alpha,\beta \in \R$ such that $a_m'=\alpha a_m +\beta b_m$ and $b_m'=\alpha b_m-\beta a_m$. \\
    \underline{If $a_m' \we b_m'=0$ and if $a_m\we b_m \neq 0$:} Thanks to an adapted choice of $\alpha'$ and $\beta'$, $a_m'\cos \la m,\cdot \ra\in \mathfrak{L}_f$. Moreover $a_m\we b_m \neq 0$, so there exist $\alpha,\beta$ such that 
    \begin{equation*}
        \left\{\begin{array}{ll}
              \alpha a_m +\beta b_m=a_m' \\
              (\alpha b_m -\beta a_m)\we a_m' \neq 0.
        \end{array}\right.
    \end{equation*}
    But $a_m'\we b_m'=0$, so $(\alpha b_m-\beta a_m) \we b_m'\neq 0$. According to \eqref{T3 prop1 rotation formulas},
    \begin{equation}
        \left\{\begin{array}{ll}
             a_m'\cos \la m,\cdot \ra + (\alpha b_m -\beta a_m)\sin \la m,\cdot \ra \in \mathfrak{L}_f, \\
              a_m'\cos \la m,\cdot \ra +b_m'\sin \la m,\cdot \ra \in \mathfrak{L}_f,
        \end{array}\right.
    \end{equation}
    So $(\alpha b_m-\beta a_m -b_m')
    \sin \la m,\cdot \ra\in \mathfrak{L}_f$ and $(\alpha b_m -\beta a_m -b_m')\we a_m' \neq 0$. \\
    \underline{If $a_m'\we b_m'\neq0$ and $a_m \we b_m \neq 0$:} There exist $\alpha,\beta$ such that $\alpha a_m+\beta b_m =a_m'$ and then necessarily $\alpha b_m -\beta a_m \neq b_m'$. As explained in the previous case, 
    \begin{equation*}
        (\alpha b_m -\beta a_m -b_m')\sin \la m,\cdot \ra \in \mathfrak{L}_f.
    \end{equation*}
    If $(\alpha b_m-\beta a_m)\we b'_m=0$, then $b_m' \sin \la m,\cdot \ra \in \mathfrak{L}_f$, and by linear combination we also have $a_m'\cos \la m,\cdot \ra$, which is sufficient because $a_m'\we b_m' \neq 0$. Else, there exist $\gamma,\delta$ such that 
    \begin{equation*}
         \gamma b_m-\delta a_m=\alpha b_m-\beta a_m -b_m',
    \end{equation*}
    and 
    \begin{equation*}
        \left\{\begin{array}{ll}
             (\gamma a_m+\delta b_m)\cos \la m,\cdot \ra + (\gamma b_m-\delta a_m)\sin \la m,\cdot \ra \in \mathfrak{L}_f, \\
              (\alpha b_m-\beta a_m -b_m')\sin \la m,\cdot \ra \in \mathfrak{L}_f,
        \end{array}\right.
    \end{equation*}
    so $(\gamma a_m +\delta b_m)\cos \la m,\cdot \ra\in \mathfrak{L}_f$ and necessarily $(\gamma a_m+\delta b_m)\we (\gamma b_m -\delta a_m)\neq 0$, because $a_m \we b_m \neq 0$. 
    \end{proof}

\begin{prop} \label{generation of modes : from 2 non colinear modes}
    Lets $m,n\in \M_f$ be such that $m\we n \neq 0$. If there exist $a_m,a_m'\in m^{\perp}$ and $c_n,d_n \in n^{\perp}$ such that 
    \begin{align*}
        a_m\cos \la m,\cdot \ra \in \mathfrak{L}_f, \quad a_m'\cos \la m,\cdot \ra \in \mathfrak{L}_f, \quad
        c_n \cos \la n,\cdot \ra +d_n \sin \la n,\cdot \ra \in \mathfrak{L}_f,
    \end{align*}
    and such that $\mathrm{span}\left\{a_m,a_m',c_n\right\}=\R^3$, then $p_k\cos\la k,\cdot \ra + q_k\sin \la k,\cdot \ra \in \mathfrak{L}_f$ for every $p_k,q_k \in k^{\perp}$ and $k \in \la m,n\ra$, where $\la m,n\ra$ denotes the subgroup of $\Z^3$ generated by $m,n$.
\end{prop}

\begin{proof}
     The vectors $\left\{a_m,a_m',c_n\right\}$ generate $\R^3$ and $\mathrm{span}\left\{a_m,a_m'\right\}=m^{\perp}$, so $c_n \we (m\we n)\neq 0$. According to \cref{theorem : appendix T3 span1}, for every $\alpha,\beta \in \R$,
     \begin{equation*}
         (\alpha c_n + \beta d_n)\cos \la n,\cdot \ra + (\alpha d_n - \beta c_n)\sin \la n,\cdot \ra \in \mathfrak{L}_f.
     \end{equation*} Or $c_n\we d_n=0$, and so we can chose $\alpha,\beta$ such that $\alpha d_n -\beta c_n=0$, or $c_n\we d_n \neq 0$, and so we can chose $\alpha,\beta$ such that
     \begin{equation*}
         (\alpha d_n -\beta c_n)\we (m \we n)=0.
     \end{equation*}
     So without loss of generality we can assume that
     \begin{equation}
         c_n \cos \la n,\cdot \ra + d_n \sin \la n,\cdot \ra \in \mathfrak{L}_f,
     \end{equation}
     with $c_n \we (m \we n)\neq 0$ and $d_n \in \mathrm{span} \left\{m \we n\right\}$. Then according to the third formula of \cref{formulas T3} applied with $(a_m,b_m=0,c_n,d_n)$, and because $\la m,d_n\ra=0$,
     \begin{equation*}
         p_{m+n}\cos\la m+n,\cdot\ra +q_{m+n}\sin \la m+n,\cdot \ra \in \mathfrak{L}_f,
     \end{equation*}
     with 
    \begin{equation*}
    \left\{\begin{array}{ll}
         p_{m+n}=-\la n,a_m \ra d_n \in \mathrm{span} \left\{m \we n\right\} \\
        q_{m+n}= -\la m,c_n \ra a_m + \la n,a_m \ra c_n.
    \end{array}\right.  
    \end{equation*}
    Thanks to the same formula applied with $(a_m',b_m'=0,c_n,d_n)$,
    \begin{equation*}
        p_{m+n}'\cos\la m+n,\cdot\ra +q_{m+n}'\sin \la m+n,\cdot \ra \in \mathfrak{L}_f,
    \end{equation*}
    with
     \begin{equation*}
    \left\{\begin{array}{ll}
         p_{m+n}'=-\la n,a_m' \ra d_n \in \mathrm{span} \left\{m \we n\right\} \\
        q_{m+n}'= -\la m,c_n \ra a_m' + \la n,a_m' \ra c_n.
    \end{array}\right.  
    \end{equation*}
    Because $\la m,c_n\ra\neq 0$ and $\mathrm{span} \left\{a_m,a_m',c_n\right\}=\R^3$, then necessarily $q_{m+n}\we q_{m+n}'\neq 0$. Because $p_{m+n}\we p_{m+n}'=0$, then necessarily for every $\alpha,\beta \in \R$,
    \begin{equation*}
        (p_{m+n}',q_{m+n}')\neq (\alpha p_{m+n}+\beta q_{m+n},\alpha q_{m+n}-\beta p_{m+n}).
    \end{equation*}
    According to \cref{adding modes : different angles}, 
    \begin{equation}
        a_{m+n}\cos \la m+n,\cdot \ra+b_{m+n}\sin \la m+n,\cdot \ra \in \mathfrak{L}_f,
    \end{equation}
    for every $a_{m+n},b_{m+n}\in (m+n)^{\perp}$. So by iteration of this computation we can obtain that $a_k\cos\la k,\cdot \ra \in \mathfrak{L}_f$ for every $a_k\in k^{\perp}$ and $k\in m\N+n\N$. Moreover, it is also verified that
    \begin{equation*}
        a_m \cos \la -m,\cdot \ra \in \mathfrak{L}_f, \quad a_m'\cos \la -m, \cdot \ra \in \mathfrak{L}_f, \quad c_n\cos \la -n,\cdot \ra - d_n\sin \la -n,\cdot \ra \in \mathfrak{L}_f,
    \end{equation*}
    so in the same way we obtain that $a_k \cos \la k,\cdot \ra \in \mathfrak{L}_f$ for every $a_k \in k^{\perp}$ and $k \in m\Z+n\Z$.
    \end{proof}

\begin{prop} \label{generation of every modes dimV0=3 dim I0=3}
    If $\mathrm{span}\left\{a_m,b_m \mid m \in \M_f \right\}=\R^3$ and $\mathrm{span}\M_f = \R^3$, and if there exists a non-zero mode $\ell \in \Gamma$, $p_{\ell},p_\ell'\in \ell^{\perp}$ such that $p_\ell\we p_\ell'\neq 0$ and such that 
    \begin{equation*}
        \left\{\begin{array}{ll}
             p_\ell \cos \la \ell,\cdot \ra \in \mathfrak{L}_f,\\
             p_\ell'\cos \la \ell,\cdot \ra \in \mathfrak{L}_f,
        \end{array}\right.
    \end{equation*}
    then $p_k\cos \la k,\cdot \ra +q_k \sin \la k,\cdot \ra \in \mathfrak{L}_f$, for every $p_k,q_k\in k^{\perp}$ and $k \in \Gamma$.
\end{prop}

\begin{proof}
    Because $\mathrm{span}\left\{a_m,b_m \mid m \in \M_f\right\}=\R^3$ and $\mathrm{span}\M_f=\R^3$, there exist a non-zero mode $m \in \M_f$ and $a_m,b_m\in m^{\perp}$ such that 
    \begin{equation*}
        a_m \cos \la m,\cdot \ra +b_m \sin \la m,\cdot \ra \in \mathfrak{L}_f,
    \end{equation*}
    and $\mathrm{span} \left\{a_m,p_\ell,p_\ell'\right\}=\R^3$. So according to \cref{generation of modes : from 2 non colinear modes}, $p_k \cos \la k,\cdot \ra \in \mathfrak{L}_f$ for every $p_k\in k^{\perp}$ and $k \in m\Z+\ell \Z$. Because $a_m \notin \ell^{\perp}$ and $\la m,a_m \ra=0$, then $m\we \ell \neq 0$. Because $\mathrm{span} \M_f=\R^3$, there exist a non-zero mode $\M_f$, and $a_n,b_n\in n^{\perp},a_n\neq0$ such that 
    \begin{equation*}
        a_n\cos \la n,\cdot \ra +b_n \sin \la n,\cdot \ra \in \mathfrak{L}_f,
    \end{equation*}
    and $\mathrm{span} \left\{m,\ell,n\right\}=\R^3$. If $\la m,a_n\ra = \la \ell, a_n \ra =0$, then $a_n =0$. If we assume without loss of generality that $a_n \notin m^{\perp}$, according to \cref{generation of modes : from 2 non colinear modes}, $p_k\cos \la k,\cdot \ra \in \mathfrak{L}_f$ for every $p_k\in k^{\perp}$ and $k \in m\Z+n\Z$. By iteration, we obtain that $p_k \cos \la k,\cdot \ra \in \mathfrak{L}_f$ for every $p_k\in k^{\perp}$ and $k \in m\Z+n \Z+\ell \Z$. Let $k' \in \Gamma$ be an other mode and $a_{k'},b_{k'}\in {k'}^{\perp},a_{k'}\neq 0$ such that 
    \begin{equation*}
        a_{k'}\cos \la k',\cdot \ra + b_k'\sin \la k',\cdot \ra \in  \mathfrak{L}_f.
    \end{equation*}
    Necessarily there exists one mode in $\left\{m,n,\ell\right\}$, for example $m$, such that $a_{k'}\notin m^{\perp}$. Then $k'\we m \neq 0$ and $p_{k}\cos \la k,\cdot \ra \in \mathfrak{L}_f$ for every $p_k \in k^{\perp}$ and $k \in k'\Z + m\Z$. So in particular for every $k \in \M_f$, $p_k \cos \la k,\cdot \ra \in \mathfrak{L}_f$ for every $p_k \in k^{\perp}$. Let $k_1,k_2 \in \M_f$ be two non-zero modes. If $k_1\we k_2\neq0$, according to the second formula of \cref{formulas T3},
    \begin{equation}
        (\la k_1,a_{k_2}\ra a_{k_1}-\la k_2,a_{k_1}\ra a_{k_2})\cos \la k_1+k_2, \cdot \ra \in \mathfrak{L}_f,
    \end{equation}
    for every $a_{k_1}\in k_1^{\perp}, a_{k_2}\in k_2^{\perp}$. So it is clear that $p_{k_1+k_2}\cos \la k_1+k_2,\cdot \ra \in \mathfrak{L}_f$ for every $p_{k_1+k_2}\in (k_1+k_2)^{\perp}$. If $k_1\we k_2=0$, then there exists $m \in \M_f$ such that $m\we k_1\neq 0$ and such that for every $a\in k_1^{\perp}$ and $b \in m^{\perp}$, 
    \begin{equation*}
        \left\{ \begin{array}{ll}
             (\la m,a\ra b-\la k_1,b\ra a)\cos \la k_1+m,\cdot\ra \in \mathfrak{L}_f,   \\
             (\la -m,a\ra b-\la k_2,b\ra a) \cos \la k_2-m, \cdot \ra \in \mathfrak{L}_f,  
        \end{array}\right.
    \end{equation*}
    and so we also prove that $p_k \cos \la k,\cdot \ra \in \mathfrak{L}_f$ for every $p_k\in k^{\perp}$ and $k \in \la k_1,k_2 \ra$. This procedure can be generalized by recurrence and we obtain that $p_k \cos \la k,\cdot \ra \in \mathfrak{L}_f$ for every $p_k\in k^{\perp}$ and $k \in  \Gamma$. 
    \end{proof}

\begin{proof}[Proof of \cref{theorem : algebra dim3 span3}.]
    If $\mathrm{span}\left\{a_m,b_m \mid m \in \M_f\right\}=\R^3$ and $\mathrm{span}\M_f=\R^3$, according to \cref{generation of every modes dimV0=3 dim I0=3}, it is sufficient to prove that there exists a non-zero mode $\ell \in \Gamma$ such that $p_\ell \cos \la \ell,\cdot \ra \in \mathfrak{L}_f $ for every $p_{\ell}\in \ell^{\perp}$. There are two possible situations: 
    \begin{enumerate}
        \item[1.] Either there exist $m,n\in \M_f$ such that $m\we n \neq 0$, 
        \begin{equation*}
            a_j\cos \la j,\cdot \ra +b_j \sin \la j,\cdot \ra \in \mathfrak{L}_f, \quad j\in \left\{m,n\right\},
        \end{equation*}
        and such that $\mathrm{span} \left\{a_m,b_m,a_n\right\}=\R^3$. 
        \item[2.] Or there exist $m,n,k\in \M_f$ such that $\mathrm{span} \left\{m,n,k\right\}=\R^3$, 
        \begin{equation}
            a_j \cos \la j,\cdot \ra +b_j \sin \la j,\cdot \ra \in \mathfrak{L}_f, \quad a_j \we b_j =0, \quad j \in \left\{m,n,k\right\},
        \end{equation}
        and such that $\mathrm{span} \left\{a_m,a_n,a_k\right\}=\R^3$.
    \end{enumerate}
    Let us consider both cases. \\
    \textbf{1.} Because $\mathrm{span} \left\{a_m,b_m,a_n\right\}=\R^3$, necessarily $a_m\we b_m \neq 0$. First we assume that $a_n \we b_n \neq 0$. According to \cref{theorem : appendix T3 span1}, for every $\alpha,\beta\in \R$,
    \begin{equation*}
        (\alpha a_j+\beta b_j)\cos \la j,\cdot \ra +(\alpha b_j-\beta a_j)\sin \la j ,\cdot \ra \in \mathfrak{L}_f \quad j \in \left\{m,n\right\}.
    \end{equation*}
    Thanks to an adapted choice of $\alpha,\beta$, we can assume that $b_m,b_n \in \mathrm{span} \left\{m\we n\right\}$. According to the third formula of \cref{formulas T3} applied with $(a_m,b_m,a_n,b_n)$, 
    \begin{equation*}
        p_{m+n}\cos \la m+n,\cdot \ra+q_{m+n}\sin \la m+n,\cdot \ra \in \mathfrak{L}_f, 
    \end{equation*}
    with 
    \begin{equation*}
        \left\{\begin{array}{ll}
             p_{m+n}=\la m,a_n\ra b_m  - \la n,a_m\ra b_n  \in \mathrm{span} \left\{m\we n\right\} \\
             q_{m+n}=- \la m,a_n \ra a_m + \la n,a_m \ra a_n. 
        \end{array}\right.
    \end{equation*}
    But it is also true that $a_n\cos \la -n,\cdot \ra-b_n \sin \la -n,\cdot\ra \in \mathfrak{L}_f$, so again we can apply the third formula of \cref{formulas T3} with $(p_{m+n},q_{m+n},a_{n},-b_n)$, and we obtain that $a_m'\cos \la m,\cdot \ra +b_m' \sin \la m,\cdot \ra \in \mathfrak{L}_f$, with 
    \begin{equation*}
    \left\{\begin{array}{ll}
         a_m'=-\la m,a_n \ra ^2 a_m  \\
         b_m'=-\la m,a_n \ra ^2 b_m + \la n,a_m \ra \la m,a_n \ra b_n.  
    \end{array}\right.
    \end{equation*}
    By assumption $a_n \we b_n \neq 0$, so $b_n \neq 0$. Moreover, $a_m,a_n\notin \mathrm{span} \left\{m \we n\right\}$, so $\la n,a_m \ra \la m,a_n \ra \neq 0$. Then $(a_m',b_m')\neq (\alpha a_m + \beta b_m, \alpha b_m -\beta a_m)$ for every $\alpha,\beta \in \R$ and $a_m,a_m' \neq 0$, so according to \cref{adding modes : different angles}, $p_m \cos \la m,\cdot \ra \in \mathfrak{L}_f$ for every $p_m \in m^{\perp}$. If we assume that $a_n \we b_n =0$, then $a_n \cos \la n,\cdot \ra \in \mathfrak{L}_f$. We can still assume that $b_m \we (m\we n)=0$, and because $\mathrm{span} \left\{a_m,b_m,a_n\right\}=\R^3$, necessarily $a_n \we b_m \neq 0$. Finally we obtain that
     \begin{equation*}
        p_{m+n}\cos \la m+n,\cdot \ra+q_{m+n}\sin \la m+n,\cdot \ra \in \mathfrak{L}_f, 
    \end{equation*}
    with 
    \begin{equation*}
        \left\{\begin{array}{ll}
             p_{m+n}=\la m,a_n\ra b_m   \\
             q_{m+n}=- \la m,a_n \ra a_m + \la n,a_m \ra a_n. 
        \end{array}\right.
    \end{equation*}
    But $m \we (m+n)\neq 0$, $a_m\we b_m \neq 0$, $p_{m+n}\we q_{m+n}\neq 0$ and $\mathrm{span} \left\{p_{m+n},q_{m+n},a_m\right\}=\R^3$, so according to the previous computations, $a_{m+n}\cos \la m+n,\cdot \ra \in \mathfrak{L}_f$ for every $a_{m+n}\in (m+n)^{\perp}$. 
    \\ \\
    \textbf{2.} If $a_j \cos \la j,\cdot \ra +b_j \sin \la j,\cdot \ra \in \mathfrak{L}_f$ and $a_j \we b_j=0$, then according to \cref{adding modes : different angles} $a_j \cos \la j,\cdot \ra \in \mathfrak{L}_f$. So in this case there exist $m,n,k\in \M_f$ such that $\mathrm{span} \left\{m,n,k\right\}=\R^3$ and such that
    \begin{equation*}
        a_m \cos \la m,\cdot \ra \in \mathfrak{L}_f, \quad a_n \cos \la n,\cdot \ra \in \mathfrak{L}_f, \quad a_k \cos \la k,\cdot \ra \in \mathfrak{L}_f,   
    \end{equation*}
    with $\mathrm{span} \left\{a_m,a_n,a_k\right\}=\R^3$. According to the second formulas of \cref{formulas T3}, $p_{m+n}\cos \la m+n,\cdot \ra\in \mathfrak{L}_f$ and $p_{n+k}\cos \la n+k,\cdot \ra \in \mathfrak{L}_f$ with 
    \begin{equation*}
        p_{m+n}= \la m,a_n \ra a_m - \la n,a_m \ra a_n, \quad p_{n+k}=\la n,a_k \ra a_n - \la k,a_n \ra a_k.
    \end{equation*}
    If $\la m,a_n\ra =0$ then $a_n \in \mathrm{span} \left\{ m\we n \right\}$ and necessarily $\la n,a_m\ra \neq 0$. So $p_{m+n}\neq 0$ and by symmetry $p_{n+k}\neq 0$.
    Then according to the same formulas, 
    \begin{equation*}
        p_{m,n,k}\cos \la m+n+k,\cdot \ra \in \mathfrak{L}_f, \quad p_{n,k,m}\cos \la n,k,m, \cdot \ra \in \mathfrak{L}_f,
    \end{equation*}
    with
    \begin{equation*}
        \left\{\begin{array}{ll}
             p_{m,n,k}=\la k,p_{m+n}\ra a_k -\la m+n,a_k \ra p_{m+n},  \\
             p_{n,k,m}=\la m,p_{n+k} \ra a_m - \la n+k,a_m\ra p_{n+k}.
        \end{array}\right.
    \end{equation*}
    If $\la m+n ,a_k \ra=0$ and $\la k,p_{m+n}\ra=0$, then $a_k,p_{m+n}\in \mathrm{span} \left\{(m+n)\we k\right\}$, and so $a_k \we p_{m+n}=0$. But $p_{m+n}\in \mathrm{span} \left\{a_m,a_n\right\}$ and $\mathrm{span} \left\{a_m,a_n,a_k\right\}=\R^3$, so necessarily $\la m+n,a_k \ra \neq 0$ or $\la k,p_{m+n}\ra \neq 0$ and $p_{m,n,k}\neq 0$. By symmetry $p_{n,k,m}\neq 0$. By computation we obtain that 
    \begin{equation*}
        \left\{\begin{array}{ll}
             p_{m,n,k}=-\la m+n,a_k\ra\la m,a_n \ra a_m + \la m+n,a_k\ra \la n,a_m \ra a_n + (\la m,a_n\ra \la k,a_m\ra -\la n,a_m \ra \la k,a_n\ra)a_k ,  \\ \\
             p_{n,k,m}=(\la m,a_n \ra \la n,a_k \ra-\la m,a_k \ra \la k,a_n \ra)a_m -\la n+k,a_m \ra \la n,a_k \ra a_n +\la n+k,a_m \ra \la k,a_n \ra a_k. 
        \end{array}\right.
    \end{equation*}
    If $p_{m,n,k}\we p_{n,k,m}=0$, then there exist a non-zero $\lambda\in \R$ such that 
    \begin{equation*}
        \left\{\begin{array}{ll}
             -\la m+n,a_k \ra \la m,a_n \ra = \lambda (\la m,a_n \ra \la n,a_k \ra - \la m,a_k \ra \la k,a_n \ra )  \\
              \la m+n,a_k \ra \la n,a_m \ra = -\lambda \la n+k,a_m \ra \la n,a_k \ra \\
              \la m,a_n \ra \la k,a_m \ra - \la n,a_m \ra \la k,a_n \ra = \lambda \la n+k,a_m \ra \la k,a_n \ra.
        \end{array}\right.
    \end{equation*}
    And so 
    \begin{equation*}
        \lambda \la n+k,a_m \ra\la n,a_k \ra \la m,a_n \ra = \lambda \la n,a_m \ra  (\la m,a_n \ra \la n,a_k \ra - \la m,a_k \ra \la k,a_n \ra ),
    \end{equation*}
    which leads to 
    \begin{equation} \label{eq : T3 condition colinearity}
        \la k,a_m \ra \la n,a_k \ra \la m,a_n \ra = - \la n,a_m \ra \la m,a_k \ra \la k,a_n \ra.
    \end{equation}
    Lets us prove that this cannot be verified in this case. By computation we obtain that 
    \begin{align*}
        \det\begin{pmatrix}
            a_m^1 &a_m^2 & a_m^3 \\
            a_n^1 & a_n^2 & a_n^3\\
            a_k^1 & a_k^2 & a_k^3
        \end{pmatrix}\begin{pmatrix}
            m^1 & n^1 &k^1 \\
            m^2 & n^2 &k^2 \\
            m^3 & n^3 &k^3 
        \end{pmatrix}&=\det\begin{pmatrix}
            0& \la n,a_m\ra & \la k,a_m\ra \\
            \la m,a_n \ra & 0 & \la k,a_n \ra \\
            \la m,a_k \ra & \la n,a_k \ra & 0
        \end{pmatrix} \\
        &=\la n,a_m\ra \la k,a_n \ra \la m,a_k \ra + \la k,a_m \ra \la m,a_n \ra \la n,a_k \ra.
    \end{align*}
    Equation \eqref{eq : T3 condition colinearity} is verified if and only if this determinant is equal to zero, which is impossible because $\mathrm{span}\left\{a_m,a_n,a_k\right\}=\mathrm{span}\left\{m,n,k\right\}=\R^3$. Then necessarily $p_{m,n,k}\we p_{n,k,m}\neq 0$. 
    
    In both cases we conclude thanks to \cref{generation of modes : from 2 non colinear modes}.
    \end{proof}

\noindent \textbf{Acknowledgments.} The authors wish to thank Mario Sigalotti and Eugenio Pozzoli for enlightening discussions and very useful remarks. This project was partly supported by the SMAI through the BOUM grant for young researchers.

\end{document}